\documentclass[10pt]{article}
\usepackage{a4wide}
\usepackage{amsmath,amssymb,amsthm}
\usepackage[mathscr]{eucal}
\usepackage{graphics}
\usepackage{epsfig}
\usepackage{psfrag}

\newcommand{\BIGOP}[1]{\mathop{\mathchoice%
{\raise-0.22em\hbox{\huge $#1$}}%
{\raise-0.05em\hbox{\Large $#1$}}{\hbox{\large $#1$}}{#1}}}
\newcommand{\bigtimes}{\BIGOP{\times}}

\newcommand{\nucl}{\mathcal{N}}
\newcommand{\orb}{\mathcal{O}}

\newcommand{\Q}{{\rm Q}}
\newcommand{\CPi}{{\rm C}\Pi}
\newcommand{\OS}{{\rm OS}}
\newcommand{\Sym}{{\rm Sym}}

\newcommand{\avt}{\mathcal{A}}
\newcommand{\Aut}{\mathop{\rm Aut}\nolimits}
\newcommand{\T}{T}
\newcommand{\FRG}{{\rm RAut}(\T)}
\newcommand{\FSG}{{\rm FAut}(\T)}

\newcommand{\pol}[1][0]{{\rm Pol}({#1})}

\newtheorem{theorem}{Theorem}
\newtheorem{proposition}[theorem]{Proposition}
\newtheorem{corollary}[theorem]{Corollary}
\newtheorem{lemma}{Lemma}
\theoremstyle{definition}
\newtheorem{definition}{Definition}
\newtheorem{example}{Example}
\newtheorem{remark}{Remark}

\begin{document}

\title{\textbf{On the conjugacy problem for finite-state automorphisms of
regular rooted trees\\ \large with an appendix by  Rapha\"el M. Jungers}}

\author{Ievgen~V.~Bondarenko, 
Natalia~V.~Bondarenko,\\ Said~N.~Sidki\footnote{The author acknowledges support from the
Brazilian Conselho Nacional de Pesquisa and from FAPDF.} ,
Flavia~R.~Zapata}

\maketitle

%
%
%

\begin{abstract}
We study the conjugacy problem in the automorphism group $\Aut(T)$ of a regular rooted tree $T$ and in its subgroup
$\FSG$ of finite-state automorphisms. We show that under the contracting condition and the finiteness of what we call
the orbit-signalizer, two finite-state automorphisms are conjugate in $\Aut(T)$ if and only if they are conjugate in
$\FSG$, and that this problem is decidable. We prove that both these conditions are satisfied by bounded automorphisms
and establish that the (simultaneous) conjugacy problem in the group of bounded automata
is decidable.\\

\noindent \textbf{Mathematics Subject Classification 2010}: 20E08, 20F10\\

\noindent \textbf{Keywords}: automorphism of a rooted tree, conjugacy problem, finite-state automorphism, finite
automaton, bounded automaton

\end{abstract}

\section{Introduction}

The interconnection between automata theory and algebra produced in the last three decades many important constructions
such as self-similar groups and semigroups, branch groups, iterated monodromy groups, self-similar (self-iterating) Lie
algebras, branch algebras, permutational bimodules, etc. (see
\cite{self_sim_groups,sid09,gns,branch,barth:branch_rings,PSZ10} and the references therein).

The connection between groups and automata occurs via a natural correspondence between invertible input-output automata
over the alphabet $X=\{1,2,\ldots ,d\}$ and automorphisms of a regular one-rooted $d$-ary tree $T$. To present this
correspondence let us index the vertices of the tree $T$ by the elements of the free monoid $X^{\ast}$, freely
generated by the set $X$ and ordered by $v\leq u$ provided $u$ is a prefix of $v$. The group $\Aut(T)$ of all
automorphisms of the tree $T$ decomposes as the permutational wreath product $\Aut(T)\cong \Aut(T)\wr \Sym(X)$, where
$\Sym(X)$ is the symmetric group on the set $X$. This decomposition allows us to represent automorphisms in the form
$g=(g|_{1},g|_{2},\ldots ,g|_{d})\pi_{g} $, where $\pi_{g}\in \Sym(X)$ is the permutation induced by the action of $g$
on the first level of the tree $T$. Iteratively, we can define the automorphism $g|_{v}=g|_{x_{1}}|_{x_{2}}\ldots
|_{x_{n}}$ for every vertex $v=x_{1}x_{2}\ldots x_{n}$ of the tree $T$, where $x_{i}\in X$. Then every automorphism
$g\in \Aut(T)$ corresponds to an input-output automaton $\avt(g)$ over the alphabet $X$ and with the set of states
$\Q(g)=\{g|_{v}\mid v\in X^{\ast}\}$. The automaton $\avt(g)$ transforms the letters as follows: if the automaton is in
state $g|_{v}$ and reads a letter $x\in X$ then it outputs the letter $y=x^{g|_{v}}$ and the state changes to
$g|_{vx}$; these operations can be best described by the labeled edge $g|_{v}\overset{x|y}{\longrightarrow}g|_{vx}$.
Following the terminology of the automata theory every automorphism $g|_{v}$ is called the \textit{state} of $g$ at
$v$.

Using this correspondence with automata one can define several classes of special subgroups of the group $\Aut(T)$. A
subgroup $G<\Aut(T)$ is called \textit{state-closed} or \textit{self-similar} if all states of every element of $G$ are
again elements of $G$. Self-similar groups play an important role in modern geometric group theory, and have
applications to diverse areas of mathematics. In particular, self-similar groups are connected with fractal geometry
through limit spaces and also with dynamical systems through iterated monodromy groups as developed by V.~Nekrashevych
\cite{self_sim_groups}. The set theoretical union of all finitely generated self-similar subgroups in $\Aut(T)$ is a
countable group denoted by $\FRG$ called the \textit{group of functionally recursive automorphisms} \cite{bru-sid}.

Automorphisms of the tree $T$ which correspond to finite-state automata are called \textit{finite-state}. More
precisely, an automorphism $g\in \Aut(T)$ is finite-state if the set of its states $\Q(g)$ is finite. The set of all
finite-states automorphisms forms a countable group denoted by $\FSG$. Every finite-state automorphism is functionally
recursive, and hence the group $\FSG $ is a subgroup of $\FRG$.

Other natural subgroups of $\Aut(T)$ are the groups $\pol[n]$ of polynomial automata of degree $n$ for every $n\geq -1$
and their union $\mathrm{Pol}({\infty})=\cup_{n}\pol[n]$. These groups were introduced by S.~Sidki in
\cite{sidki:circ}, who tried to classify subgroups of $\FSG$ by the cyclic structure of the associated automata and by
the growth of the number of paths in the automata avoiding the trivial state. Especially important is the \textit{group
$\pol[0]$ of bounded automata} whose elements are called bounded automorphisms. A finite-state automorphism $g$ is
\textit{bounded} if the number of paths of length $m$ in the automaton $\avt(g)$ avoiding the trivial state is bounded
independently of $m$. It is to be noted that most of the studied self-similar groups are subgroups of $\pol[0]$. In
particular, the Grigorchuk group \cite{grig}, the Gupta-Sidki group \cite{gupta-sidki}, the Basilica \cite{grig-zuk2}
and BSV groups \cite{bsv}, the finite-state spinal groups \cite{branch}, the iterated monodromy groups of
post-critically finite polynomials \cite{self_sim_groups}, and many others, are generated by bounded automorphisms.
Moreover, it is shown in \cite{bondnek:pcf} that finitely generated self-similar subgroups of $\pol[0]$ are precisely
those finitely generated self-similar groups whose limit space is a post-critically finite self-similar set which play
an important role in the development of analysis on fractals (see \cite{kigami:anal_fract}).

In this paper we consider the conjugacy problem and the order problem in the groups $\Aut(T)$, $\FRG$, $\FSG$,
$\pol[0]$. It is well known that the word problem is solvable in the group $\FSG$ and hence in all its subgroups, while
it is an open problem in the group $\FRG$. Furthermore, the order and conjugacy problems are open in $\FSG$ and $\FRG$.
The conjugacy classes of the group $\Aut(T)$ were described in \cite{sid98,gawr:conju}. It is not difficult to
construct two finite-state automorphisms which are conjugate in $\Aut(T)$ but not conjugate in $\FSG$ (see \cite{gns}).
At the same time, two finite-state automorphisms of finite order are conjugate in $\Aut(T)$ if and only if they are
conjugate in $\FSG$ (see \cite{russev}). The conjugacy classes of the group $\mathrm{Pol}({-1})$ of finitary
automorphisms were determined for the binary tree in \cite{bru-sid} and for the general case in \cite{marcio-tese}.

The conjugacy problem was solved for some well-known finitely generated subgroups of $\pol[0]$. In
particular, the solution of the conjugacy problem in the Grigorchuk group was given in
\cite{leonov,rozhkov98}, and it was generalized in \cite{wil-zal,grig-wil} to certain classes of branch
groups and their subgroups of finite index. Moreover, it was shown in \cite{lmu} that the conjugacy problem
in the Grigorchuk group is decidable in polynomial time. The conjugacy problem for the Basilica and BSV
groups was treated in \cite{grig-zuk2}. A finitely generated self-similar subgroup of $\FSG$ with
unsolvable conjugacy problem was constructed in a recent preprint \cite{SV10}.

The general approach in considering any algorithmic problem dealing with automorphisms of the tree $T$ is to reduce the
problem to some property of their states. The order and the conjugacy problems lead us to the following definition. For
an automorphism $a\in \Aut(T)$ consider the orbits $Orb_{a}(v)$ of its action on the vertices $v$ of the tree and
define the set
\begin{equation*}
\OS(a)=\left\{ a^{m}|_{v}\mid v\in X^{\ast},\text{}m=|Orb_{a}(v)|\right\}
\end{equation*}
which we call the \textit{orbit-signalizer} of $a$. It is not difficult to see that the order problem is decidable for
finite-state automorphisms with finite orbit-signalizers. We prove that every bounded automorphism has finite
orbit-signalizer and hence the order problem is decidable for bounded automorphisms.

\begin{proposition}
The order problem for bounded automorphisms is decidable.
\end{proposition}

We treat the conjugacy problem firstly in the group $\Aut(T)$. Given two automorphisms $a,b\in \Aut(T)$ we
construct a \textit{conjugator graph} $\Psi (a,b)$ based on the sets $\OS(a),\OS(b)$, which portrays the
inter-dependence among the different conjugacy subproblems encountered in trying to find a conjugator for
the pair $a,b$, and which leads to the construction of a conjugator if it exists.

\begin{theorem}
Two finite-state automorphisms $a,b$ with finite orbit-signalizers are conjugate in $\Aut(T)$ if and only if they are
conjugate in $\FRG$ if and only if the conjugator graph $\Psi (a,b)$ is nonempty.
\end{theorem}

An important class of self-similar groups are contracting groups. This property for groups corresponds to the expanding
property in a dynamical system. A finitely generated self-similar group is contracting if the length of its elements
asymptotically contracts when applied to their states. A finite-state automorphism is called \textit{contracting} if
the self-similar group generated by its states is contracting. Bounded automorphisms are contracting (see
\cite{bondnek:pcf}), however in contrast to bounded automorphisms, contracting automorphisms do not form a group. For
contracting automorphisms with finite orbit-signalizers, we prove that conjugation is controlled by the group of
finite-state automorphisms.

\begin{theorem}
Two contracting automorphisms with finite orbit-signalizers are conjugate in $\Aut(T)$ if and only if they are
conjugate in $\FSG$.
\end{theorem}

We prove a number of results for the conjugacy problem for bounded automorphisms in Section~4, which we
collect in the following theorem.

\begin{theorem}
\begin{enumerate}
\item The (simultaneous) conjugacy problem for bounded automorphisms in $\Aut(T)$ is decidable.

\item Two bounded automorphisms are conjugate in the group $\Aut(T)$
if and only if they are conjugate in the group $\FSG$.

\item The (simultaneous) conjugacy problem in $\pol[0]$ is
decidable.

\item Two bounded automorphisms are conjugate in the group $\mathrm{Pol}({\infty})$ if and
only if they are conjugate in the group $\pol[0]$.
\end{enumerate}
\end{theorem}

We develop two algorithms for the solution of the conjugacy problem in the group $\pol[0]$. The first one
exploits the cyclic structure of bounded automorphisms. While the second exploits the number of active
states of bounded automorphisms. This last counting argument translates to a bounded trajectory problem for
nonnegative matrices which is shown to be decidable in the appendix by Rapha\"el M. Jurgens. The methods
developed in this study provide a construction for possible conjugators whenever the associated conjugacy
problems are solved.

The last section presents some examples, which illustrate the solution of
the conjugacy problems, and describes the connection between the property of
having finite orbit-signalizers and other properties of automorphisms.

\section{Preliminaries}

The set $X^{*}$ is considered as the set of vertices of the tree $T$ as described in Introduction. The
length of a word $v=x_1x_2\ldots x_n\in X^{*}$ for $x_i\in X$ is denoted by $|v|=n$. The set $X^n$ of words
of length $n$ forms the \textit{$n$-th level} of the tree $T$. The vertices $X^{*}$ are ordered by the
lexicographic order on words induced by the order on the set $X$.

We are using right actions, so the image of a vertex $v\in X^{*}$ under the action of an automorphism $g\in\Aut(T)$ is
written as $v^{g}$ or $(v)g$, and hence $v^{g\cdot h}=(v^g)^h$.

The state $g|_v$ of $g$ at $v$, which was defined in Introduction, is the unique automorphism of the tree $T$ such that
the equality $(vw)^g=v^g(w)^{g|_v}$ holds for all words $w\in X^{*}$. Computation of states of automorphisms is done as
follows:
\begin{align*}
(g\cdot h)|_{v}=g|_{v}\cdot h|_{(v)g}, \quad g^{-1}|_{v}=(g|_{(v)g^{-1}})^{-1}, \quad g^{n}|_{v}=g|_{v}\cdot
g|_{(v)g}\cdot \ldots \cdot g|_{(v)g^{n-1}}
\end{align*}
for all $g,h\in\Aut(T)$ and $v\in X^{*}$. Therefore, conjugation is computed by the rules
\begin{eqnarray*}
\left( h^{-1}gh\right) |_{v} &=&\left( h^{-1}\right) |_{v}\left( gh\right)
|_{\left( v\right) h^{-1}}=\left( h|_{\left( v\right) h^{-1}}\right)
^{-1}g|_{\left( v\right) h^{-1}}h|_{\left( v\right) h^{-1}g}; \\
\left( h^{-1}gh\right) |_{\left( v\right) h} &=&\left( h|_{v}\right)
^{-1}g|_{v}h|_{\left( v\right) g},
\end{eqnarray*}
and if $(v)g=v$ then
\begin{equation*}
\left( h^{-1}gh \right)|_{\left( v\right) h}=\left(
h|_{v}\right)^{-1}g|_{v}h|_{v}.
\end{equation*}

The multiplication of two automorphisms expressed  as $g=(g|_{1},g|_{2},\ldots ,g|_{d})\pi_{g}$,
$h=(h|_{1},h|_{2},\ldots ,h|_{d})\pi_{h}$ is performed by the rule
\begin{equation*}
g\cdot h=(g|_{1}h|_{\left( 1\right) g},g|_{2}h|_{\left( 2\right) g},\ldots ,g|_{d}h|_{\left( d\right) g})\pi_{g}\pi
_{h}.
\end{equation*}
Every permutation $\pi \in \Sym(X)$ can be identified with the automorphism $(e,e,\ldots ,e)\pi $ of the tree $T$
acting on the vertices by the rule $(xv)^{\pi}=x^{\pi}v$ for $x\in X$ and $v\in X^{\ast}$.

The group $\FRG$ of functionally recursive automorphisms consists of automorphisms which can be constructed as follows.
A finite set of automorphisms $g_1,g_2,\ldots,g_m$ is called \textit{functionally recursive} if there exist words
$w_{ij}$ over $\{g_1^{\pm 1},g_2^{\pm 1},\ldots,g_m^{\pm 1}\}$ and permutations $\pi_i\in\Sym(X)$ such that
\begin{align*}
g_1&= (w_{11},w_{12},\ldots,w_{1d})\pi_1 \\
g_2&= (w_{21},w_{22},\ldots,w_{2d})\pi_2 \\
& \ \ \vdots \\
g_m&= (w_{m1},w_{m2},\ldots,w_{md})\pi_m.
\end{align*}
This system has a unique solution in the group $\Aut(T)$%
, here the action of each element $g_i$ on the first level of the tree $T$
is given by the permutation $\pi_i$, and the action of the state $g_i|_j$ is
uniquely defined by the word $w_{ij}$. An automorphism of the tree is called
\textit{functionally recursive} provided it is an element of some
functionally recursive set of automorphisms.

For an automorphism $g\in \Aut(T)$ define the numerical sequence
\begin{equation*}
\theta_{k}(g)=|\{v\in X^{k}:g|_{v}\mbox{ acts non-trivially on }X\}|\ %
\mbox{ for }\ k\geq 0,
\end{equation*}
which describes the activity growth of $g$. Looking at the asymptotic behavior of the sequence $\theta_{k}(\cdot)$ we
can define different classes of automorphisms of the tree $T$.

The elements $g\in\Aut(T)$, whose sequence $\theta_k(g)$ is eventually zero, are called \textit{finitary
automorphisms}. In other words, an automorphism $g$ is finitary if there exists $k$ such that $g|_v=1$ for all $v\in
X^k$, and the smallest $k$ with this property is called the \textit{depth} of $g$. The set of all finitary
automorphisms forms a group denoted by $\mathrm{Pol}({-1})$.

For a finite-state automorphism $g\in \FSG$ the sequence $\theta_{k}(g)$ can grow either exponentially or
polynomially (see \cite[Corollary~7]{sidki:circ}). The set of all finite-state automorphisms $g\in \FSG$,
whose sequence $\theta _{k}(g)$ is bounded by a polynomial of degree $n $, is \textit{the group $\pol[n]$
of polynomial automata of degree $n$}. In the case $n=0$, when the sequence $\theta_{k}(g)$ is bounded,
then the automorphism $g$ is called \textit{bounded} and the group $\pol[0]$ is called \textit{the group of
bounded automata}. We get an ascending chain of subgroups $\pol[n]<\mathrm{Pol}({n+1})$ for $n\geq -1$. The
union $\mathrm{Pol}({\infty})=\cup_{n}\pol[n]$ is called \textit{the group of polynomial automata}. If we
replace the condition \textquotedblleft $g|_{v}$ acts non-trivially on $X$" by \textquotedblleft $g|_{v}$
is non-trivial" in the definition of the sequence $\theta_{k}(\cdot)$ then we still get the same groups
$\pol[n]$.

The bounded and polynomial automorphisms can be characterized by the cyclic structure of their automata as
described in \cite{sidki:circ}. A cycle in an automaton is called \textit{trivial} if it is a loop at the
state corresponding to the trivial automorphism. Then an automorphism $g\in \FSG$ is polynomial if and only
if any two different non-trivial cycles in the automaton $\avt(g)$ are disjoint. Moreover, $g\in \pol[n]$,
when $n$ is the largest number of non-trivial cycles connected by a directed path. In particular, an
automorphism $g\in \FSG$ is bounded if and only if any two different non-trivial cycles in the automaton
$\avt(g)$ are disjoint and not connected by a directed path. We say that $g$ is \textit{circuit} if there
exists a non-empty word $v\in X^{\ast}$ such that $g=g|_{v}$, i.e. $g$ lies on a cycle in the automaton
$\avt(g)$. If $g$ is a circuit bounded automorphism then the state $g|_{v}$ is finitary for every word $v$,
which is not read along the circuit.

\section{Conjugation in groups of automorphisms of the tree}

\label{Section Conj}

Let us recall the description of the conjugacy classes in the group $\Aut(T)$.

\textbf{Conjugacy classes in $\Aut(T)$}. First, recall that every conjugacy class of the symmetric group $\Sym(X)$ has
a unique (left-oriented) representative of the form
\begin{equation}
(1,2,\ldots ,n_{1})(n_{1}+1,n_{1}+2,\ldots ,n_{2})\ldots
(n_{k-1}+1,n_{k-1}+2,\ldots ,n_{k}),  \label{eqn_conjug_repre_Sym(X)}
\end{equation}
where $1\leq n_{1}\leq n_{2}-n_{1}\leq \ldots \leq n_{k}-n_{k-1}$ and $n_{k}=d=|X|$. This observation can
be generalized to the group $\Aut(T)$ (see \cite[Section~4.1]{sid98}). Given an automorphism
$a=(a|_{1},a|_{2},...,a|_{d})\pi_{a}$ in $\Aut(T)$ we can conjugate it to a unique (left-oriented)
representative of its conjugacy class using the following basic steps.

1. Conjugate the permutation $\pi_a\in\Sym(X)$ to its unique left-oriented conjugacy representative
(\ref{eqn_conjug_repre_Sym(X)}).

2. Consider every cycle $\tau_i=(n_i+1,n_i+2,\ldots,n_{i+1})$ in the
representative (\ref{eqn_conjug_repre_Sym(X)}) of $\pi_a$ and define
\begin{equation*}
h_{i+1}=\left(a|_{n_i+1},e,(a|_{n_i+2}) ^{-1},(a|_{n_i+2} a|_{n_i+3})^{-1},...,(a|_{n_i+2} a|_{n_i+3}\ldots
a|_{n_{i+1}-1})^{-1}\right).
\end{equation*}
Conjugate $a$ by the the automorphism $h=(h_1,h_2,\ldots,h_k)$ to obtain $h^{-1}ah=(a_1,a_2,\ldots,a_k)$, where
\begin{equation*}
a_i=\left(e,...,e,a|_{n_i+2}\ldots a|_{n_{i+1}} a|_{n_i+1}\right)\tau_i.
\end{equation*}

3. Apply the steps 1 and 2 to the automorphisms $a|_{n_i+2}\ldots
a|_{n_{i+1}} a|_{n_i+1}$.

It is direct to see that an infinite iteration of this procedure produces a well-defined automorphism of the tree which
conjugates $a$ into a representative and that two different representatives are not conjugate in $\Aut(T)$.

Another approach is based on the fact that two permutations are conjugate if and only if they have the same
cycle type. The \textit{orbit type} of an automorphism $a\in \Aut(T)$ is the labeled graph, whose vertices
are the orbits of $a$ on $X^{\ast }$, every orbit is labeled by its cardinality, and we connect two orbits
$\mathcal{O}_{1}$ and $\mathcal{O}_{2}$ by an edge if there exist vertices $v_{1}\in \mathcal{O}_{1}$ and
$v_{2}\in \mathcal{O}_{2}$, which are adjacent in the tree $T$. Then two automorphisms of the tree $T$ are
conjugate if and only if their orbit types are isomorphic as labeled graphs (see \cite[Theorem
3.1]{gawr:conju}). In particular it follows, that the group $\Aut(T)$ is ambivalent (that is, every element
is conjugate with its inverse). More generally, every automorphism $a\in \Aut(T)$ is conjugate with
$a^{\xi}$ for every unit $\xi $ of the ring $\mathbb{Z}_{m}$ of $m$-adic integers, where $m$ is the
exponent of the group $\Sym(X)$ (see \cite[Section~4.3]{sid98}).\label{page_orbit_type}

\vspace{0.2cm}\textbf{Conjugation lemma}. We say that an element $h$ is a \textit{conjugator for the pair}
$(a,b)$ if $h^{-1}ah=b$, and we use the notation $h:a\rightarrow b$. For $a,b\in \Aut(T)$ and the
permutations $\pi_{a},\pi_{b}\in \Sym(X)$ induced by the action of $a$ and $b$ on $X$, the set of
permutational conjugators for the pair $(\pi_{a},\pi _{b})$ is denoted by
\begin{equation*}
\CPi (a,b)=\{\pi \in \Sym(X):\pi ^{-1}\pi_{a}\pi =\pi_{b}\}
\end{equation*}
(this set can be empty).

The study of the conjugacy problem in the automorphism groups of the tree $T$
is based on the following standard lemma.

\begin{lemma}
\label{lemma_conju_problem} Let $a,b,h\in\Aut(T)$.

\begin{enumerate}
\item If $h^{-1}ah=b$ then $|Orb_{a}(v)|=|Orb_{b}(v^{h})|$ for every $v\in
X^{\ast}$ and
\begin{equation*}
(h|_{v})^{-1}a^{m}|_{v}h|_{v}=b^{m}|_{v^{h}},
\end{equation*}
where $m=|Orb_{a}(v)|$.

\item Let $\mathcal{O}_{1},\mathcal{O}_{2},\ldots ,\mathcal{O}_{k}$ be the
orbits of the action of $a$ on $X$. If there exists $\pi\in\CPi(a,b)$ such that $a^{|\orb_i|}|_{v}$ and
$b^{|\orb_i|}|_{v^{\pi}}$ are conjugate in $\Aut(T)$ for every $i=1,2,\ldots,k$, where $v\in\orb_i$ is an arbitrary
point, then $a$ and $b$ are conjugate in $\Aut(T)$.
\end{enumerate}
\end{lemma}

\begin{proof}
The first statement follows from the equalities $h^{-1}a^mh=b^m$, $(v){a^m}=v$.

Let $Orb_a(v)=\{v_0=v, v_1,\ldots, v_{m-1}\}$, where $v_i=(v)a^i$. Put $u=v^h$, then $Orb_b(u)=\{u_0=u, u_1,\ldots,
u_{m-1}\}$, where $u_i=(u)b^i$ and $u_i=v_i^h$. Then
\begin{eqnarray*}
b|_{u_0} &=& (h^{-1}ah)|_{u_0} = (h|_{v_0})^{-1} a|_{v_0} h|_{v_1} \\
b|_{u_1} &=& (h^{-1}ah)|_{u_1} = (h|_{v_1})^{-1} a|_{v_1} h|_{v_2} \\
\ldots & &\ldots\\
b|_{u_{m-1}} &=& (h^{-1}ah)|_{u_{m-1}} = (h|_{v_{m-1}})^{-1} a|_{v_{m-1}} h|_{v_0}
\end{eqnarray*}
Multiplying these equations, we get
\begin{align*}
(h|_v)^{-1} a^m|_v h|_v&=(h|_{v_0})^{-1} \left(a|_{v_0} a|_{v_1}\ldots a|_{v_{m-1}}\right) h|_{v_0}=\\
&=b|_{u_0} b|_{u_1}\ldots b|_{u_{m-1}}=b^m|_{u}.
\end{align*}
In particular
\begin{align}
(h|_{v_i})^{-1} a^m|_{v_i} h|_{v_i}&=(h|_{v_i})^{-1} \left(a|_{v_i} a|_{v_{i+1}}\ldots a|_{v_{i-1}}\right) h|_{v_i}=\nonumber\\
&=b|_{u_i} b|_{u_{i+1}}\ldots b|_{u_{i-1}}=b^m|_{u_i},\nonumber\\
h|_{v_i}=\left(a|_{v_0}\ldots a|_{v_{i-2}}a|_{v_{i-1}}\right)^{-1}&h|_{v_0}\left(b|_{u_0}\ldots
b|_{u_{i-2}}b|_{u_{i-1}}\right)=(a^i|_{v})^{-1} h|_v b^i|_{u}.\label{eqn_in_Conjugacy_Lemma}
\end{align}
\end{proof}

If $a$ and $b$ are finite-state automorphisms \emph{(we need this only for the word problem)},
Lemma~\ref{lemma_conju_problem} suggests a branching decision procedure for the conjugacy problem in $\Aut(T)$. We call
this procedure by \textrm{CP} and remark that it may not stop in general.

\vspace{0.2cm}\textbf{The order problem in $\Aut(T)$}. The problem of finding the order of a given element
of $\Aut(T)$ can be handled in a manner similar to the above. The next observation gives a simple condition
used in many papers to prove that an automorphism has infinite order.

\begin{lemma}
\label{lemma_finite_order} Let $a\in\Aut(T)$.

\begin{enumerate}
\item Let $\mathcal{O}_{1},\mathcal{O}_{2},\ldots ,\mathcal{O}_{k}$ be the
orbits of the action of $a$ on $X$. Define $a_{i}=a^{m_{i}}|_{x_{i}}$ for every $i=1,2,\ldots ,k$, where
$m_{i}=|\mathcal{O}_{i}|$ and $x_{i}\in \mathcal{O}_{i}$ is an arbitrary point. The automorphism $a$ has finite order
if and only if all the states $a_{i}$ have finite order. Moreover, in this case, the order of $a$ is equal to
\begin{equation*}
|a|=\emph{lcm}(m_{1}|a_{1}|,m_{2}|a_{2}|,\ldots ,m_{k}|a_{k}|).
\end{equation*}

\item Suppose $a_{i}=a$ for some choice of $x_{i}\in \mathcal{O}_{i}$. If $m_{i}>1$ then $a$ has infinite order. If $m_{i}=1$ then $a$ has finite order if and only if $a_{j}$ has finite order
for all $j\neq i$, in which case we can remove the term $m_{i}|a_{i}|$ from the right hand side of the above equality.
\end{enumerate}
\end{lemma}

If $a$ is a finite-state automorphism, then the word problem $a_{i}=a$ can be effectively solved and
Lemma~\ref{lemma_finite_order} suggests a branching procedure to find the order of a. We call this procedure by OP and
remark that it may not stop in general. Such a procedure is implemented in the program packages \cite{FR,AutomGrp}.

\vspace{0.2cm} \textbf{Orbit-signalizer.} Lemmas \ref{lemma_conju_problem} and \ref{lemma_finite_order}
lead us to define the \textit{orbit-signalizer} of an automorphism $a\in\Aut(T)$ as the set
\begin{equation*}
\OS(a)=\left\{ a^{m}|_{v}\mid v\in X^{\ast},\text{ } m=|Orb_{a}(v)|\right\},
\end{equation*}
which contains all automorphisms that may appear in the procedures \textrm{OP} and \textrm{CP}. Notice that
if $m=|Orb_{a}(v)|$, $l=|Orb_{a}(vx)|$, and $k=|Orb_{a^m|_v}(x)|$ then $l=mk$ and
\begin{equation}\label{eqn_state_of_a_state_in_orbit_sign}
a^l|_{vx}=\left( a^m|_v \right)^k|_x.
\end{equation}
This observation implies the recursive procedure to find the set $\OS(a)$. We start from the set
$\OS_{0}(a)=\{a\}$ and compute consecutively
\[
\OS_{n+1}(a)=\left\{ b^m|_x\mid b\in \OS_n(a), x\in X, m=|Orb_b(x)| \right\}.
\]
Then $\OS(a)=\cup_{n\geq 0} \OS_n(a)$. It follows from construction that if $\OS_{n+1}(a)$ does not contain
new elements, i.e., $\OS_{n+1}(a)\subset \cup_{i=0}^n\OS_i(n)$, then we can stop and
$\OS(a)=\cup_{i=0}^n\OS_i(n)$. In particular, if the set $\OS(a)$ is finite, then this procedure stops in
finite time and we can find $\OS(a)$ algorithmically. For automorphisms with finite orbit-signalizers one
can model the procedures OP and CP by finite graphs.

\vspace{0.2cm} \textbf{Order graph.} Consider an automorphism $a\in \Aut(T)$ which has finite
orbit-signalizer. We construct a finite graph $\Phi (a)$ with the set of vertices $\OS(a)$, called the
\textit{order graph} of $a$, which models the branching procedure OP. The edges of this graph are
constructed as follows. For every $b\in \OS(a)$ consider all orbits $\mathcal{O}_{1},\mathcal{O}_{2},\ldots
,\mathcal{O}_{k}$ of the action of $b$ on $X$ and let $x_{i}\in \mathcal{O}_{i}$ be the least element in
$\mathcal{O}_{i}$. It is easy to see that $b^{m_{i}}|_{x_{i}}\in \OS(a)$ for $m_{i}=|\mathcal{O}_{i}|$, and
we introduce the labeled edge $b\overset{m_{i}}{\longrightarrow}b^{m_{i}}|_{x_{i}}$ in the graph $\Phi (a)$
for every $i=1,\ldots ,k$. Then Lemma~\ref{lemma_finite_order} can be reformulated as follows.

\begin{proposition}
Let $a\in\Aut(T)$ have finite orbit-signalizer. Then $a$ has finite order if and only if all edges in the
directed cycles in the order graph $\Phi(a)$ are labeled by~$1$.
\end{proposition}

Moreover, in this case we can compute the order of $a$ using the graph $\Phi(a)$. Remove all the edges of
every directed cycle in $\Phi(a)$. Then the only dead vertex of $\Phi(a)$, i.e. the vertex without outgoing
edges, is the trivial automorphism, which has order $1$. Then inductively, for $b\in\OS(a)$ consider all
outgoing edges from $b$, and let $m_1,m_2,\ldots,m_k$ be the edge labels and $b_1,b_2,\ldots,b_k$ be the
corresponding end vertices, whose order we already know. Then by Lemma~\ref{lemma_finite_order} the order
of $b$ is equal to the least common multiple of $m_i|b_i|$. We illustrate the construction of the order
graph and the solution of the order problem in Example~\ref{ex_order_problem} of
Section~\ref{Section_Examples}.

\vspace{0.2cm}\textbf{Conjugator graph.} Consider automorphisms $a,b\in \Aut(T)$ both of which have finite
orbit-signalizers. We construct a finite graph $\Psi (a,b)$, called the \textit{conjugator graph} of the pair $(a,b)$,
modeled after the branching procedure CP of Lemma~\ref{lemma_conju_problem}. The vertices of the graph $\Psi (a,b)$ are
the triples $(c,d,\pi)$ for $c\in\OS(a)$, $d\in \OS(b)$, and $\pi\in\CPi(c,d)$ whenever this last set is nonempty. The
edges are constructed as follows.

Let $\mathcal{O}_{i}\left( c\right) $ for $1\leq i\leq k$ be the orbits of $c $ in its action on $X$ and let
$x_{i}\left( c\right) $ denote the least element in each $\mathcal{O}_{i}\left( c\right) $. We will simplify the
notation by writing instead $\mathcal{O}_{i}$ and $x_{i}$ with the understanding that these refer to $c\in \OS(a)$
under consideration.

For any vertex $(c,d,\pi)$, if one of the sets $\CPi (c^{m}|_{x_{i}},d^{m}|_{{x_{i}}^{\pi}})$ with
$m=|\mathcal{O}_{i}|$ is empty, then the triple $(c,d,\pi)$ is a dead vertex. 
Otherwise we introduce in the graph the edge
\begin{equation*}
(c,d,\pi)\overset{x_{i}}{\longrightarrow}(c^{m}|_{x_{i}},d^{m}|_{{x_{i}}^{\pi}},\tau)\ \mbox{ with }
m=|\mathcal{O}_{i}|
\end{equation*}
for every $\tau \in \CPi (c^{m}|_{x_{i}},d^{m}|_{{x_{i}}^{\pi}})$ and $i=1,\ldots ,k$. Notice that $c^{m}|_{x_{i}}\in
\OS(a)$, $d^{m}|_{x_{i}^{\pi}}\in \OS(b)$, and hence the triple $(c^{m}|_{x_{i}},d^{m}|_{{x_{i}}^{\pi}},\tau)$ is
indeed a vertex of the graph.

We simplify the graph obtained above using the following reductions. Remove the vertex $(c,d,\pi)$ which does not have
an outgoing edge labeled by $x_{i}$ for some $i$. Also, remove all edges leading to these deleted vertices. We repeat
the reductions as long as possible to reach the graph $\Psi (a,b)$.

If the graph $\Psi (a,b)$ is empty, then the automorphisms $a$ and $b$ are not conjugate. Otherwise they
are conjugate and every conjugator $h:a\rightarrow b$ can be constructed level by level as follows. Choose
any vertex $(a,b,\pi)$ in $\Psi (a,b)$ and define the action of $h$ on the first level by $x^{h}=x^{\pi}$
for $x\in X$. There is an outgoing edge from $(a,b,\pi)$ labeled by $x_{i}=x_{i}(a)$, as explained
previously. Choose an edge for every $x_{i}$ and let $(c_{i},d_{i},\pi_{i})$ be the corresponding end
vertex. We define the action of the state $h|_{x_{i}}$ by the rule $\left( x\right)
^{h|_{x_{i}}}=x^{\pi_{i}}$ for $x\in X$. All the other states of $h$ on the vertices of the first level are
uniquely defined by Equation~(\ref{eqn_in_Conjugacy_Lemma}) at the end of the proof of
Lemma~\ref{lemma_conju_problem}, and thus we get the action of $h$ on the second level. Similarly, we
proceed further with the vertices $(c_{i},d_{i},\pi_{i})$ and construct the action of $h$ on the third
level, and so on. Notice that even if the same vertex $(c,d,\pi)$ appears at different stages of the
definition of $h$ we still have a freedom to choose different outgoing edges from $(c,d,\pi )$ in each
stage of the construction.

\vspace{0.2cm}\textbf{Basic conjugators.} Let us construct certain conjugators, called \textit{basic conjugators} for
the pair $(a,b)$, by making as few choices as possible, in the sense that if we arrive at a triple $(c,d,\pi)$ at some
stage of the construction then we choose the same permutation $\pi \in \CPi (c,d)$ whenever the pair $(c,d)$ reappears
further down. That is, for every two vertices $(c,d,\pi_{1})$ and $(c,d,\pi_{2})$ obtained under construction we insist
to have $\pi_{1}=\pi_{2}$. More precisely every basic conjugator can be defined using special subgraphs of the
conjugator graph $\Psi (a,b)$. Consider the subgraph $\Gamma $ of $\Psi (a,b)$, which satisfies the following
properties:

\begin{enumerate}
\item The subgraph $\Gamma $ contains some vertex $(a,b,\pi)$ for $\pi \in
\CPi (a,b)$.

\item For every vertex $(c,d,\pi)$ of $\Gamma$ and every letter $x_{i}$,
there exists precisely one outgoing edge from $(c,d,\pi)$ labeled by $x_{i}$. In particular, the graph is
deterministic, and the number of outgoing edges at the vertex $(c,d,\pi)$ of the graph $\Gamma $ is equal to the number
of orbits of $c$ on $X$.

\item For every $c\in \OS(a)$ and $d\in \OS(b)$ there is at most
one vertex of the form $(c,d,\ast)$ in the graph $\Gamma$. In other words, if $(c,d,\pi_{1})$ and $(c,d,\pi_{2})$ are
vertices of $\Gamma $ then $\pi_{1}=\pi_{2}$.
\end{enumerate}

If the graph $\Psi (a,b)$ is nonempty, there always exist subgraphs of $\Psi (a,b)$, which satisfy the properties 1-3.
For every such a subgraph $\Gamma $ we construct the basic conjugator $h=h(\Gamma)$ as follows. We construct a
functionally recursive system involving every conjugator $h_{(c,d)}:c\rightarrow d$, where $(c,d,\pi)$ is a vertex of
$\Gamma $ and thus in particular, we construct $h=h_{(a,b)}$. First, we define the action of the conjugator $h_{(c,d)}$
on the first level by the rule $x^{h_{(c,d)}}=x^{\pi} $ for $x\in X$, where the permutation $\pi \in \CPi (c,d)$ is
uniquely defined such that the triple $(c,d,\pi)$ is a vertex of $\Gamma $. For every edge $(c,d,\pi
)\overset{x}{\longrightarrow}(c^{\prime},d^{\prime},\pi ^{\prime})$ we define the states of the conjugator $h_{(c,d)}$
on the letters from the orbit $\mathcal{O}=\{x,(x)c,(x)c^{2},\ldots ,(x)c^{m-1}\} $ of $x$ under $c$ recursively by the
rule
\begin{equation*}
h_{(c,d)}|_{x}=h_{(c^{\prime},d^{\prime})}\quad \mbox{ and }\quad h_{(c,d)}|_{(x)c^{i}}=\left( c^{i}|_{x}\right)
^{-1}\cdot h_{(c^{\prime},d^{\prime})}\cdot d^{i}|_{x^{\pi}},
\end{equation*}
for $i=1,\ldots ,m-1$. These rules completely define the automorphisms $h_{(c,d)}$. By Lemma~\ref{lemma_conju_problem}
every constructed automorphism $h_{(c,d)}$ is indeed a conjugator for the pair $(c,d)$. Since the graph $\Gamma $ is
finite, and the automorphisms $a,b$ are finite-state, we get a functionally recursive system which uniquely defines the
basic conjugator $h=h_{(a,b)}$ given by the subgraph $\Gamma$.

We have proved the following theorem.

\begin{theorem}
\label{thm_conj_fso_conj_gr} Let $a,b\in \FSG$ have finite orbit-signalizers, and let $\Psi (a,b)$ be the corresponding
conjugator graph. Then $a$ and $b$ are conjugate in $\Aut(T)$ if and only if they are conjugate in $\FRG$ if and only
if the graph $\Psi (a,b)$ is nonempty.
\end{theorem}

In particular, the conjugacy problem for finite-state automorphisms with finite orbit-signalizers is
decidable in the groups $\Aut(T)$ and $\FRG$. We present examples of the construction of the conjugator
graph and basic conjugators in Example~\ref{ex_conjugacy_problem} of Section~\ref{Section_Examples}.

The same method works for the simultaneous conjugacy problem, which given automorphisms $a_{1},a_{2},\ldots ,a_{k}$ and
$b_{1},b_{2},\ldots ,b_{k}$ asks for the existence of an automorphism $h$ such that $h^{-1}a_{i}h=b_{i}$ for all $i$.
We again consider the permutations $\pi $ such that $\pi ^{-1}\pi_{a_{i}}\pi =\pi_{b_{i}}$ for all $i$, take an orbit
$\mathcal{O}_{i}$ of $a_{i}$ on $X$, let $x_{i}\in \mathcal{O}_{i}$ be the least element and $m_{i}=|\mathcal{O}_{i}|$.
Then the problem reduces to the simultaneous conjugacy problem for $a^{m_{1}}|_{x_{1}},a^{m_{2}}|_{x_{2}},\ldots
,a^{m_{k}}|_{x_{k}}$ and $b^{m_{1}}|_{x_{1}^{\pi}},b^{m_{2}}|_{x_{2}^{\pi}},\ldots ,b^{m_{k}}|_{x_{k}^{\pi}}$. If
automorphisms $a_{i}$ and $b_{i}$ are finite-state and have finite orbit-signalizers, then we can similarly construct
the associated conjugator graph so that Theorem~\ref{thm_conj_fso_conj_gr} holds.

\vspace{0.2cm} \textbf{Conjugation of contracting automorphisms.} A self-similar subgroup $G<\Aut(T)$ is
called \textit{contracting} if there exists a finite set $\mathcal{N}\subset G$ with the property that for
every $g\in G$ there exists $n\in\mathbb{N}$ such that $g|_v\in\mathcal{N}$ for all words $v$ of length
$\geq n$. The smallest set $\mathcal{N}$ with this property is called the \textit{nucleus} of the group. An
automorphisms $f\in\Aut(T)$ is called \textit{contracting} if the self-similar group generated by all
states of $f$ is contracting. It follows from the definition that contracting automorphisms are
finite-state.

\begin{theorem}
\label{theor_contr_conjug} Two contracting automorphisms $a,b\in\Aut(T)$ with finite orbit-signalizers are conjugate in
the group $\Aut(T)$ if and only if they are conjugate in the group $\FSG$.
\end{theorem}

\begin{proof}
We will prove that all the basic conjugators for the pair $(a,b)$ are finite-state. We need a few lemmata, which are
interesting in themselves.

\begin{lemma}\label{lemma_contr_finite_number_values}
Let $G$ be a contracting self-similar group, and let $H$ be a finite subset of $G$. Then the set of all possible
elements of the form
\begin{align*}
(\ldots ((h_1|_{x_1}\cdot h_2)|_{x_2}\cdot h_3)|_{x_3}\cdot\ldots\cdot h_n)|_{x_n},\quad (h_n\cdot\ldots
\cdot(h_3\cdot (h_2\cdot h_1|_{x_1})|_{x_2})|_{x_3}\ldots)|_{x_n},
\end{align*}
where $h_i\in H$ and $x_i\in X$, is finite.
\end{lemma}
\begin{proof}
The statement is a reformulation of Proposition~2.11.5 in \cite{self_sim_groups}. We sketch the proof for completeness.

We can assume that the set $H$ is self-similar, i.e. $h|_v\in H$ for all $h\in H$ and $v\in X^{*}$ (all the elements
are finite-state), and contains the nucleus $\nucl$ of the group $G$. There exists a number $k$ such that
$H^2|_v\subset\nucl\subset H$ for all words $v$ of length $\geq k$. Then $H^{2n}|_v\subset H^n$ for all $v\in X^k$ and
$n\geq 1$. It is sufficient to prove that there are finitely many elements of the form
\[
(\ldots ((h_1|_{v_1}\cdot h_2)|_{v_2}\cdot h_3)|_{v_3}\cdot\ldots\cdot h_n)|_{v_n}
\]
for $h_i\in H^k$ and $v_i\in X^k$. Then $h_1|_{v_1}\in H^k$ and $(h_1|_{v_1}\cdot h_2)|_{v_2}\in H^{2k}|_{v_2}\subset
H^k$. Inductively we get that all the above elements belong to $H^k$.
\end{proof}

The next lemma is similar to Corollary~2.11.7 in \cite{self_sim_groups}, which states that different self-similar
contracting actions with the same virtual endomorphism are conjugate via a finite-state automorphism.

\begin{lemma}\label{lemma_h_wreath_rec_finite_state}
Let $G$ be a contracting self-similar group and $H\subset G$ be a finite subset. Suppose that the
automorphism $g\in\Aut(T)$ satisfies the condition that for every $x\in X$ there exist $h,h'\in H$ such
that $g|_x=hgh'$. Then $g$ is finite-state.
\end{lemma}
\begin{proof}
For an arbitrary word $x_0x_1x_2\ldots x_n\in  X^{*}$ we have
\begin{align}\label{eqn_g|_v_in_lemma_finite_state}
g|_{x_0x_1x_2\ldots x_n}&=(h_1 g h'_1)|_{x_1x_2\ldots x_n}=\bigl((((h_1|_{x_1}\cdot h_2)|_{x_2}\cdot
h_3)|_{x_3}\cdot\ldots\cdot h_n)|_{x_n}\cdot h_{n+1}\bigr)\cdot g \cdot\nonumber\\
&\quad\cdot \bigl(h'_{n+1}\cdot(h'_n\cdot\ldots \cdot(h'_3\cdot (h'_2\cdot
h'_1|_{y_1})|_{y_2})|_{y_3}\ldots)|_{y_n}\bigr),
\end{align}
where $h_i,h'_i$ are some elements in $H$, and $y_i\in X$. By Lemma~\ref{lemma_contr_finite_number_values} the above
products assume a finite number of values, and hence $g$ is finite-state.
\end{proof}

Notice that in the previous lemma we only need that the groups generated by all states of $h_i$ and
separately by all states of $h'_i$ be contracting, while together they may generate a non-contracting
group.

\begin{lemma}\label{lemma_system_wreath_rec_finite_state}
Let $F\subset \Aut(T)$ be a finite collection of automorphisms. Suppose that there exist two contracting self-similar
groups $G_1$, $G_2$ and finite subsets $H_1\subset G_1$, $H_2\subset G_2$ with the condition that for every $g\in F$
and every letter $x\in X$ there exist $h_1\in H_1$, $h_2\in H_2$ and $g'\in F$ such that $g|_{x}=h_1g'h_2$. Then all
the automorphisms in $F$ are finite-state.
\end{lemma}
\begin{proof}
The proof is the same as in Lemma~\ref{lemma_h_wreath_rec_finite_state}. The only difference is that on the right hand
side of Equation~(\ref{eqn_g|_v_in_lemma_finite_state}) instead of $g$ may appear any element of the finite set $F$.
\end{proof}

To finish the proof of Theorem~\ref{theor_contr_conjug} it is sufficient to notice that all the basic
conjugators for the pair $(a,b)$ satisfy Lemma~\ref{lemma_system_wreath_rec_finite_state}, and hence all of
them are finite-state.
\end{proof}

Example~\ref{ex_contr_finiteOS_counter_ex} of Section~\ref{Section_Examples} shows that we cannot drop the assumption
about orbit-signalizers in the theorem.

The finiteness of orbit-signalizers can be used to prove that certain automorphisms are not conjugate in the group
$\FSG$.

\begin{proposition}
\label{prop_fso_necessary} Let $f,g\in \Aut(T)$ be conjugate in $\FSG$. Then $f$ has finite orbit-signalizer if and
only if $g$ does.
\end{proposition}

\begin{proof}
Let $h^{-1}fh=g$ for a finite-state automorphism $h$, and suppose $f$ has finite orbit-signalizer. Then
$m=|Orb_f(v)|=|Orb_g(v^h)|$ for every $v\in X^{*}$ and
\[
g^m|_{v^h}=(h|_v)^{-1}f^m|_vh|_v\in (h|_v)^{-1}\OS(f)h|_v. 
\]
It follows that $\OS(g)\subset \Q(h)^{-1} \OS(f) \Q(h)$, where $\Q(h)$ is the set of states of $h$, and
hence the set $\OS(g)$ is finite.
\end{proof}

\section{Conjugation of bounded automorphisms}

How to check that a given finite-state automorphism has finite orbit-signalizer is yet another algorithmic
problem. Let us show that some classes of automorphisms have finite orbit-signalizers. Every finite-state
automorphism $a$ of finite order has finite orbit-signalizer. Here the set $\OS(a)$ is bounded by the
number of all states of all powers of $a$, which is finite. In particular, if a finite-state automorphism
has infinite orbit-signalizer, then it has infinite order.

\begin{proposition}
\label{prop_bound_fso_condition} Every bounded automorphism has finite
orbit-signalizer.
\end{proposition}

\begin{proof}
Let $a$ be a bounded automorphism, and choose a constant $C$ so that the number of non-trivial states
$a|_v$ for $v\in X^n$ is not greater than $C$ for every $n\geq 0$. Then for every vertex $v\in X^{*}$ the
state $a^m|_v$ with $m=|Orb_a(v)|$ is a product of no more than $C$ states of $a$, which is a finite set.
\end{proof}

In particular, the orbit-signalizer of a bounded automorphism can be computed algorithmically, the
procedure \textrm{CP} solves the conjugacy problem for bounded automorphism in $\Aut(T)$, and the procedure
\textrm{OP} finds the order of a bounded automorphism.

\begin{corollary}
(1). The order problem for bounded automorphisms is decidable.\newline (2). The (simultaneous) conjugacy problem for
bounded automorphisms in $\Aut(T)$ is decidable.
\end{corollary}

\begin{theorem}
\label{thm_bounded_conju_in_Aut_f} Two bounded automorphisms are conjugate in the group $\Aut(T)$ if and only if they
are conjugate in the group $\FSG$.
\end{theorem}

\begin{proof}
The bounded automorphisms are contracting by \cite{bondnek:pcf} and have finite orbit-signalizers by
Proposition~\ref{prop_bound_fso_condition}, hence we can apply Theorem~\ref{theor_contr_conjug}.
\end{proof}

\begin{corollary}
\label{cor_bounded_conju_inverse} Let $a$ be a bounded automorphism. Then $a$ and $a^{-1}$ are conjugate in $\FSG$.
\end{corollary}

\begin{corollary}
Let a bounded automorphism $f$ and a contracting automorphism $g$ be conjugate in $\Aut(T)$. Then $f$ and $g$ are
conjugate in $\FSG$ if and only if $g$ has finite orbit-signalizer.
\end{corollary}

\subsection{Decision of conjugation between bounded automorphisms by a finitary automorphism}

Consider the conjugacy problem for bounded automorphisms in the group $\mathrm{Pol}({-1})$ of finitary
automorphisms. One of the approaches is to restrict the depth of a possible finitary conjugator. Let $a,b$
be two bounded automorphisms. If $a$ and $b$ are conjugate in $\mathrm{Pol}({-1})$ and $h:a\rightarrow b$
is a finitary conjugator then every state $h|_{x}$ for $x\in X$ is a finitary conjugator for the pair
$(a^{m}|_{x},b^{m}|_{x^{h}})$ with $m=|Orb_{a}(x)|$, and $h|_x$ has less depth than $h$. However, it is
possible that every pair $(a^{m}|_{x},b^{m}|_{x^{h}})$ for $x\in X$ with $m=|Orb_{a}(x)|$ is conjugate via
a finitary conjugator of depth $\leq d$, while $(a,b)$ is not conjugate via a finitary conjugator of depth
$\leq d+1$. Hence we still do not get a bound on the depth of $h$ even if we know the bound on the depth of
a finitary conjugator for every pair $(a^{m}|_{x},b^{m}|_{x^{h}})$. The problem is that we need to find a
finitary conjugator $h|_{x}$ for the pair $(a^{m}|_{x},b^{m}|_{x^{h}})$ so that all elements
$h|_{(x)a^{i}}=(a^{i}|_{x})^{-1}\cdot h|_{x}\cdot b^{i}|_{x^{h}}$ for $i=0,\ldots ,m-1$ are finitary. To
overcome this difficulty we introduce configurations of orbits, which describe these dependencies.

\vspace{0.2cm}\textbf{Configurations of orbits.} Every configuration will be of the form
$\mathcal{C}=\{(\alpha,\beta),DP_{\mathcal{C}}\}$, where $(\alpha,\beta)$ is a pair of automorphism called
the \textit{main pair} of $\mathcal{C}$, and $DP_{\mathcal{C}}$ is the set of pairs of automorphism called
\textit{dependent pairs}. Configurations for the pair $(a,b)$ are constructed recursively as follows. At
zero level we have just one configuration $\Lambda_0=\{\mathcal{C}\}$, where
$\mathcal{C}=\{(a,b),DP_{\mathcal{C}}=\{(e,e)\}\}$. Further we construct recursively the set
$\Lambda_{n+1}$ from the set $\Lambda_n$. Take a configuration $\mathcal{C}\in\Lambda_n$,
$\mathcal{C}=\{(\alpha,\beta), DP_{\mathcal{C}}\}$. Consider every orbit $\mathcal{O}$ of the action of
$\alpha$ on $X$, let $x$ be the least element in $\mathcal{O}$ and $m=|\mathcal{O}|$. For every
$\pi\in\CPi(\alpha,\beta)$ define new configuration $\mathcal{C}^{\prime}=\mathcal{C}^{\prime}_{x,\pi}=\{
(\alpha^m|_x, \beta^m|_{x^{\pi}}), DP_{\mathcal{C}^{\prime}}\}$, where
\begin{eqnarray}\label{eqn_construction_configurations}
DP_{\mathcal{C}^{\prime}}=\{ ( (\alpha^i c)|_x, (\beta^i d)|_{x^{\pi}})\, |\, (c,d)\in DP_{\mathcal{C}}
\mbox{ and } i=0,\ldots,m-1\}.
\end{eqnarray}
The set $\Lambda_{n+1}$ consists of all configurations $\mathcal{C}^{\prime}$ constructed in this way. Then
$\Lambda=\cup_{n\geq 0} \Lambda_n$ is the set of \textit{configurations} for $(a,b)$. It follows from
construction that if $\Lambda_{n+1}$ does not contain new configurations, i.e.,
$\Lambda_{n+1}\subset\cup_{i=0}^n\Lambda_i$, then we can stop and $\Lambda=\cup_{i=0}^n\Lambda_i$. In
particular, if the set $\Lambda$ is finite then it can be computed algorithmically.

\begin{lemma}
The set of configurations for a pair of bounded automorphisms is finite and can be computed
algorithmically.
\end{lemma}
\begin{proof}

Let $\mathcal{C}=\{(\alpha,\beta), DP_{\mathcal{C}}\}$ be a configuration for $(a,b)$ and denote
$A_{\mathcal{C}}=\{c\, |\, (c,d)\in DP_{\mathcal{C}}\}$. We prove by induction that there exists a word
$v\in X^{*}$ such that
\begin{equation}\label{eqn_lemma_config_and_v}
\alpha=a^l|_v\,\mbox{ and }\, A_{\mathcal{C}}=\{ a^j|_v\, |\, j=0,1,\ldots, l-1\},
\end{equation}
where $l=|Orb_a(v)|$. The basis of induction is the initial configuration
$\mathcal{C}=\{(a,b),DP_{\mathcal{C}}=\{(e,e)\}\}$ that satisfies this condition for the empty
word~$v=\emptyset$. Suppose inductively that a configuration $\mathcal{C}$ satisfies the condition for a
word $v$ and proceed with the construction of $\mathcal{C}^{\prime}$. Let $x$ be the least element in an
orbit of the action of $\alpha$ on $X$ and put $m=|Orb_{\alpha}(x)|$. Then $|Orb_{a}(vx)|=lm$ and we get
\begin{equation*}
\alpha^m|_x=(a^l|_v)^m|_x=a^{lm}|_{vx} \quad \mbox{ and }\vspace{-0.5cm}
\end{equation*}
\begin{eqnarray*}
A_{\mathcal{C}^{\prime}}&=&\{ (\alpha^ic)|_x\, |\, c\in A_{\mathcal{C}} \mbox{ and } i=0,1,\ldots,m-1 \}=\\
&=&\{ ((a^l|_v)^ia^j|_v)|_x\, |\, i=0,1,\ldots,m-1 \mbox{ and } j=0,1,\ldots,l-1\}=\\
&=&\{a^k|_{vx}\, |\, k=0,1,\ldots,lm-1\} \ (\mbox{here } k=li+j).
\end{eqnarray*}
Hence $\mathcal{C}^{\prime}$ satisfies condition (\ref{eqn_lemma_config_and_v}) for the word $vx$.

In particular, $\alpha\in\OS(a)$ and can assume only a finite number of values. The number of different
sets $\{a^{j}|_{v}\, |\, 0\leq j<|Orb_{a}(v)|\}$, $v\in X^{\ast}$, for a bounded automorphism $a$, is
finite (the proof is the same as of Proposition~\ref{prop_bound_fso_condition}). Hence there are finitely
many possibilities for the set $A_{\mathcal{C}}$. The same observation holds for $\beta$ and the set
$B_{\mathcal{C}}=\{d\, |\, (c,d)\in DP_{\mathcal{C}} \}$. It follows that the number of configurations for
a pair of bounded automorphisms is finite.
\end{proof}

\begin{remark}\label{rem_configuration_for_orbit}
Let $h^{-1}ah=b$. Let $\mathcal{O}$ be an orbit of the action of $a$ on $X^{*}$ and $v\in\mathcal{O}$ be the least element in the orbit and $m=|\mathcal{O}|$. Then $\mathcal{C}=\{(a^{m}|_{v},b^{m}|_{v^{h}}), DP_{\mathcal{C}}\}$, where
\begin{equation*}
DP_{\mathcal{C}}=\{(a^{i}|_{v},b^{i}|_{v^{h}}),\mbox{ for } i=0,\ldots ,m-1\},
\end{equation*}
is a configuration for $(a,b)$.
\end{remark}

\textbf{Configurations satisfied by a finitary automorphism.} We say that a finitary automorphism $h$
\textit{satisfies a configuration} $\mathcal{C}$ if $h$ is a conjugator for the main pair of $\mathcal{C}$
and all elements $c^{-1}hd$ for $(c,d)\in DP_{\mathcal{C}}$ are finitary. We say that a configuration
$\mathcal{C}$ \textit{has depth} $\leq d$ if $\mathcal{C}$ is satisfied by a finitary automorphism $h$ of
depth $\leq d$ and all elements $c^{-1}hd$ for $(c,d)\in DP_{\mathcal{C}}$ have depth $\leq d$. In
particular, the automorphisms $a, b$ are conjugate in $\mathrm{Pol}({-1})$ if and only if the configuration
$\mathcal{C}=\{(a,b), DP_{\mathcal{C}}=\{(1,1)\}\}$ is satisfied by a finitary automorphism. Let us show
that it is decidable whether a given configuration $\mathcal{C}$ can be satisfied by a finitary
automorphism.

\begin{lemma}\label{lemma_configur_conju_depth}
Let $\mathcal{C}=\{(\alpha,\beta), DP_{\mathcal{C}}\}$ be a configuration. Consider all orbits
$\mathcal{O}_1,\mathcal{O}_2,\ldots,\mathcal{O}_k$ of the action of $\alpha$ on $X$ and let
$x_j\in\mathcal{O}_j$ be the least element in $\mathcal{O}_j$. The configuration $\mathcal{C}$ has depth
$\leq d$ if and only if there exists $\pi\in\CPi(\alpha,\beta)$ such that every configuration
$\mathcal{C}'_{x_j,\pi}$, $j=1, \ldots,k$, has depth $\leq d-1$.
\end{lemma}
\begin{proof}
Suppose $h^{-1}\alpha h=\beta$ and all automorphisms $c^{-1}hd$, $(c,d)\in DP_{\mathcal{C}}$, are finitary
of depth $\leq d$. For the permutation $\pi$ in $\CPi(\alpha,\beta)$ we take $\pi_h$. Then
\[
(h|_{x_j})^{-1}\alpha^{m_j}|_{x_j}h|_{x_j}=\beta^{m_j}|_{x_j^{\pi}}
\]
and for every $(c,d)\in DP_{\mathcal{C}}$ and $i=0,1,\ldots,m-1$ we get
\begin{eqnarray}\label{eqn_lemma_state_of_depended_pair}
( (\alpha^i c)|_{x_j})^{-1} \cdot h|_{x_j}\cdot (\beta^i d)|_{{x_j}^{\pi}}=(c|_{(x_j)\alpha^i})^{-1}
h|_{(x_j)\alpha^i} d|_{(x_j)\alpha^i h} =(c^{-1}hd)|_y,
\end{eqnarray}
where $y=(x_j)\alpha^i c$. All automorphisms $(c^{-1}hd)|_y$ are finitary of depth $\leq d-1$. Hence every
configuration $\mathcal{C}'_{x_j,\pi}$ has depth $\leq d-1$.

Conversely, suppose there exists $\pi\in\CPi(\alpha,\beta)$ such that every $\mathcal{C}'_{x_j,\pi}$ is
satisfied by a finitary automorphism $h_j$. Define automorphism $h$ by the rules $\pi_h=\pi$ and
\[
h|_{x_j}=h_j \ \mbox{ and } \ h|_{(x_j)\alpha^i}=(\alpha^i|_{x_j})^{-1}  h|_j \beta^i|_{x_j^{\pi}}.
\]
Note that $(\alpha^i|_{x_j}, \beta^i|_{x_j^{\pi}})=( (\alpha^i c)|_{x_j}, (\beta^i d)|_{x_j^{\pi}})$ for
$(c,d)=(e,e)\in DP_{\mathcal{C}}$ and hence every pair $(\alpha^i|_{x_j}, \beta^i|_{x_j^{\pi}})$ belongs to
$DP_{\mathcal{C}^{\prime}}$. Therefore the automorphism $h$ is finitary. Also $h$ is a conjugator for
$(\alpha,\beta)$ by construction and satisfies the configuration $\mathcal{C}$ by
Equation~(\ref{eqn_lemma_state_of_depended_pair}).
\end{proof}

\begin{corollary}\label{cor_bound_conj_in_finitary}
If $a$ and $b$ are conjugate in $\mathrm{Pol}({-1})$ then there exists a finitary conjugator of depth $\leq
|\Lambda|$. In particular, the conjugacy problem for bounded automorphisms in the group
$\mathrm{Pol}({-1})$ is decidable.
\end{corollary}

Instead of just running through all finitary automorphisms with a given bound on the depth, the algorithm
can be realized as follows. Construct the set $\Lambda$ of all configuration for a given pair $(a,b)$. We
will consecutively label configurations by numbers which correspond to their depths. First, we identify
configurations of depth $0$, which are precisely configurations $\mathcal{C}=\{(\alpha,\beta),
DP_{\mathcal{C}}\}$ such that $\alpha=\beta$ and $c=d$ for all $(c,d)\in DP_{\mathcal{C}}$. Then
iteratively we label a configuration $\mathcal{C}=\{(\alpha,\beta), DP_{\mathcal{C}}\}$ by number $d$ if
there exists $\pi\in\CPi(\alpha,\beta)$ such that each $\mathcal{C}^{\prime}_{x_j,\pi}$ is already labeled
by a number $\leq d-1$. After this process finishes, the configurations labeled by $d$ can be satisfied by
a finitary automorphisms of depth $\leq d$, while the unlabeled configurations cannot be satisfied by
finitary automorphisms.

\subsection{The conjugacy problem in the group of bounded automata} In this subsection we prove that the
conjugacy problem in the group of bounded automata is decidable. We will show two approaches.

\textbf{First approach: by using cyclic structure of bounded automata.} Let $a$ and $b$ be bounded
automorphisms. We apply the following algorithm to check whether $a$ and $b$ are conjugate in $\pol[0]$.
The algorithm will consecutively determine the pairs from $\OS(a)\times\OS(b)$ that are conjugate in
$\pol[0]$. Further we prove that the algorithm is correct.

\textbf{Step 1}. Take $(c,d)\in\OS(a)\times\OS(b)$ and compute the set $\Lambda(c,d)$ of all configurations
for $(c,d)$. For every configuration $\mathcal{C}$ and every $(\gamma_1,\delta_1),(\gamma_2,\delta_2)\in
DP_{\mathcal{C}}$ check whether $(\gamma_1^{-1}\gamma_2)^{-1}c(\gamma_1^{-1}\gamma_2)$ and
$(\delta_1^{-1}\delta_2)^{-1}d(\delta_1^{-1}\delta_2)$ are conjugate in $\pol[-1]$, and if yes then $c,d$
are conjugate in $\pol[0]$. Apply this step to every pair $(c,d)\in\OS(a)\times\OS(b)$. Note that since
$\mathcal{C}=\{(c,d), DP_{\mathcal{C}}=\{(e,e)\}\}$ is a configuration for $(c,d)$ we also detect every
pair $(c,d)$ conjugated in $\pol[-1]$ (just take $\gamma_1=\gamma_2=\delta_1=\delta_2=e$).

\textbf{Step 2}. Take $(c,d)\in\OS(a)\times\OS(b)$. Consider all words $u\in X^{*}$ such that $u^c=u$ and
$|u|\leq |\OS(c)|\cdot|\OS(d)|$. Consider all circuit automorphisms $h$ such that $h|_u=h$ and every
finitary state of $h$ has depth $\leq|\Lambda(c,d)|$. Note that there are only finitely many bounded
automorphisms with these properties. For every such $h$ check whether $h^{-1}ch=d$. We apply this step to
every pair $(c,d)\in\OS(a)\times\OS(b)$ not detected in Step $1$.

\textbf{Step 3}. For every pair $(c,d)\in\OS(a)\times\OS(b)$ for which we still do not know whether it is
conjugate in $\pol[0]$ proceed as follows. Consider orbits $\orb_1,\ldots,\orb_k$ of the action of $c$ on
$X$, let $x_i\in\orb_i$ be the least element in the orbit and $m_i=|\orb_i|$. Check whether there exists
$\pi\in\CPi(c,d)$ such that for every pair $(c^{m_i}|_{x_i},d^{m_i}|_{x_i^{\pi}})$ (it belongs to
$\OS(a)\times\OS(b)$) we already know that it is conjugate in $\pol[0]$. If yes then $c,d$ are conjugate in
$\pol[0]$. We repeat this step as long as possible until no new pairs are detected. The other pairs from
$\OS(a)\times\OS(b)$ are not conjugate in $\pol[0]$.

\begin{proof}[Proof of correctness of the algorithm]

First, every pair detected in one of the steps is indeed conjugate in $\pol[0]$. We need to prove the
converse that if $a,b$ are conjugate in $\pol[0]$ then the pair $(a,b)$ will be detected. Let $h^{-1}ah=b$
for a bounded automorphism $h$. There exists a level $l$ such that for every $v\in X^{l}$ the state $h|_v$
is either circuit or finitary. Consider the orbits of the action of $a$ on $X^l$. Let $v$ be the least
element in an orbit $\orb$ and $m=|\orb|$. Then $h|_v$ is a conjugator for the pair
$(c,d)=(a^m|_v,b^m|_{v^h})\in \OS(a)\times\OS(b)$.

If $h|_v$ is finitary then the pair $(c,d)$ is detected in Step~$1$.

If $h|_v$ is circuit then we take a circuit conjugator $g$ for $(c,d)$ having a circuit of the shortest
length. Let $u$ be the word which is read along the circuit, so here $g|_{u}=g$. Now consider two cases.

Case 1: $u^c\neq u$. Then $g|_{u^c}=(c|_u)^{-1}g|_ud|_{u^g}$ is finitary. Since $g|_u=g$ we get $(g|_u)^{-1}c(g|_u)=d$ and
\[
(g|_{u^c})^{-1} \left((c|_u)^{-1} c c|_u\right) g|_{u^c}= (d|_{u^g})^{-1} d d|_{u^g}.
\]
Hence $(c|_u)^{-1} c c|_u$ and $(d|_{u^g})^{-1} d d|_{u^g}$ are conjugate in $\pol[-1]$. Let $w$ be the
least element in the orbit $\orb'=Orb_c(u)$ and $u=(w)c^i$. Let $\mathcal{C}$ be the configuration for
$(c,d)$ associated to the orbit $\orb'$ and the conjugator $g$ (see
Remark~\ref{rem_configuration_for_orbit}). Put $\gamma_1=c^i|_w$, $\gamma_2=c^{i+1}|_w$,
$\delta_1=d^i|_{w^g}$, $\delta_2=d^{i+1}|_{w^g}$ and note that $(\gamma_1,\delta_1), (\gamma_2,\delta_2)\in
DP_{\mathcal{C}}$. Then $c|_u=\gamma_1^{-1}\gamma_2$ and $d|_{u^g}=\delta_1^{-1}\delta_2$. Therefore the
pair $(c,d)$ is detected in Step~$1$.

Case 2: $u^c=u$. In this case the states of $g$ along the circuit do not have dependencies, and we have a
freedom to change these states without changing other states of the same level. Suppose there are two
states $g|_{v_{1}}$ and $g|_{v_{2}}$ along the circuit (let $|v_{1}|<|v_{2}|$ so $v_{1}$ is a prefix of
$v_{2}$), which solve the same conjugacy problem $(c|_{v_{1}},d|_{v_{1}^{g}})=(c|_{v_{2}},d|_{v_{2}^{g}})$.
Define the automorphism $f$ by the rules $f|_{v_{1}}=g|_{v_{2}}$, $f|_{v}=g|_{v}$ for $v\in
X^{|v_{1}|},v\neq v_{1}$, and the action of $f$ on $X^{|v_{1}|}$ is the same as the action of $g$. Then $f$
is a circuit bounded conjugator for the pair $(c,d)$ and it has smaller circuit length; we arrive at
contradiction. Hence, the states of $g$ along the circuit solve different conjugacy problems and therefore
$|u|\leq |\OS(c)|\cdot|\OS(d)|$. Now consider finitary states of $g$. We can assume that $g$ is chosen is
such a way that the value of $\max (\mbox{depth of $f$})$, where $f$ ranges over all finitary states of
$g$, is the smallest possible over all conjugators for $(c,d)$ with $g|_u=g$. Let $g|_w$ be a finitary
state. Then every state of $g$ along the orbit of $w$ under $c$ is finitary. Hence the configuration
$\mathcal{C}$ for $(c,d)$ that corresponds to the orbit of $w$ is satisfied by a finitary automorphism.
However the depth of $\mathcal{C}$ is not greater than $|\Lambda(c,d)|$ by
Lemma~\ref{lemma_configur_conju_depth}. Hence the depth of every state of $g$ along the orbit of $w$ is not
greater than $|\Lambda(c,d)|$. Therefore $g$ satisfies the conditions in Step~2 and the pair $(c,d)$ is
detected in this step.

We have proved that every pair from the action of $a$ on $X^l$ is detected in Steps~$1,2$. The other pairs coming from
the orbits of the action of $a$ on smaller levels, in particular the pair $(a,b)$, are detected in
Step~$3$.
\end{proof}

\textbf{Second approach: by calculation of active states.} Let $a,b$ be two bounded automorphisms. We check
whether $a,b$ are conjugate in $\Aut(T)$ and if not then they are not conjugate in $\pol[0]$. So further we
assume that $a,b$ are conjugate in $\Aut(T)$. Every conjugator for the pair $(a,b)$ can be constructed
level by level as described in Section~\ref{Section Conj} on page~10, by choosing the conjugating
permutation for every orbit of $a$. The number of orbits may grow when we pass from one level to the next,
and consequently the number of choices grows. However the number of configurations of orbits is finite, and
it is easy to see (and also follows from the previous method) that if $a$ and $b$ are conjugate in
$\pol[0]$ then there exists a bounded conjugator $h$ such that for all orbits of the same level and of the
same configuration the corresponding states of $h$ are the same. Hence it is sufficient to choose a
conjugating permutation only for configurations. We will show how to count the number of active states
depending on our choice.

Suppose we have constructed a conjugator $h$ up to the $n$-th level. Consider an orbit $\mathcal{O}$ of the
action of $a$ on $X^{n}$ and let $\mathcal{C}=\{(\alpha ,\beta);DP_{\mathcal{C}}\}$ be its configuration
(see Remark~\ref{rem_configuration_for_orbit}) and $v$ be the least element in $\mathcal{O}$. The set
$DP_{\mathcal{C}}$ remembers only the pairs $(c,d)$ which appear in the formula
$h|_{(v)a^{i}}=c^{-1}h|_{v}d$, here $c=a^{i}|_{v}$ and $d=b^{i}|_{v^{h}}$ for $i=0,\ldots
,|\mathcal{O}|-1$; we will call $h|_{(v)a^{i}}$ a state of \textit{type} $(c,d)$. However the number of
states of type $(c,d)$ is lost in this way. To preserve this information we introduce the nonnegative
integer column-vector $u_{\mathcal{C}}$ of dimension $|DP_{\mathcal{C}}|$, where $u_{\mathcal{C}}(c,d)$ for
$(c,d)\in DP_{\mathcal{C}}$ is equal to the number of $i$ such that $c=a^{i}|_{v}$ and $d=b^{i}|_{v^{h}}$.
When we pass to the next level, we choose some permutation $\pi \in \CPi (\alpha ,\beta)$ and define
$x^{h|_{v}}=x^{\pi}$ for $x\in X$. Then we check which states of $h$ on the vertices from the orbit
$\mathcal{O}$ are active and which are not: the state $h|_{(v)a^{i}}$ of type $(c,d)$ is active if the
permutation $\pi_{c^{-1}}\pi \pi_{d}$ is non-trivial. We store this information in the row-vector
$\theta_{\mathcal{C},\pi}$ of dimension $|DP_{\mathcal{C}}|$ by making $\theta _{\mathcal{C},\pi}(c,d)=1$
if $\pi_{c^{-1}}\pi \pi_{d}\neq e$ and $\theta_{\mathcal{C},\pi}(c,d)=0$ otherwise. Hence, when we choose
the permutation $\pi$ then the number of active states of $h$ along the orbit $\mathcal{O}$ is equal to
$\theta_{\mathcal{C},\pi}\cdot u_{\mathcal{C}}=\sum \theta_{\mathcal{C},\pi}(c,d)u_{\mathcal{C}}(c,d)$.

Let $\Lambda $ be the set of all configurations for the pair $(a,b)$. Let $\Pi =\bigtimes_{\mathcal{C}\in
\Lambda}\CPi _{\mathcal{C}}$, where $\CPi_{\mathcal{C}}=\CPi (\alpha ,\beta)$ and $(\alpha ,\beta)$ is the
main pair of $\mathcal{C}$. We view $\Pi $ as the set of choices, so that when we choose $\pi \in \Pi $ we
have chosen a conjugating permutation for every configuration. The sets $\Lambda $ and $\Pi $ are finite.
For $\pi =(\pi _{\mathcal{C}})_{\mathcal{C}\in \Lambda}$ define $\theta_{\pi}:=(\theta_{\mathcal{C},\pi
_{\mathcal{C}}})_{\mathcal{C}\in \Lambda}$ and construct the square nonnegative integer matrix $A_{\pi}$ of
dimension $\sum_{\mathcal{C}\in \Lambda }|DP_{\mathcal{C}}|$, where the rows and columns of $A_{\pi}$ are
indexed by pairs $[\mathcal{C},(c,d)]$ with $(c,d)\in DP_{\mathcal{C}}$. The entry of $A_{\pi}$ in the
intersection of $[\mathcal{C}_1,(c_1,d_1)]$-row and $[\mathcal{C}_2,(c_2,d_2)]$-column is calculated as
follows. Recall the construction of configurations $\mathcal{C}^{\prime}$ induced by $\mathcal{C}_1$ and
$\pi_{\mathcal{C}_1}$, and given in Equation~(\ref{eqn_construction_configurations}). Let $(\alpha,\beta)$
be the main pair of $\mathcal{C}_1$, consider orbits $\orb_1,\ldots,\orb_k$ of the action of $\alpha$ on
$X$, let $x_j$ be the least element in $\orb_j$ and $m_j=|\orb_j|$. Let
$\mathcal{C}^{\prime}_j=\mathcal{C}^{\prime}_{x_j,\pi_{\mathcal{C}_1}}$ be the induced configurations.
Define the $[\mathcal{C}_1,(c_1,d_1)]\times [\mathcal{C}_2,(c_2,d_2)]$-entry of $A_{\pi}$ as
\[
\sum_{j} \left|\left\{ 0\leq i<m_j\, |\, \mathcal{C}^{\prime}_j=\mathcal{C}_2 \ \mbox{ and } \
((\alpha^ic_1)|_{x_j}, (\beta^id_1)|_{x_j^{\pi_{\mathcal{C}_1}}})=(c_2,d_2) \right\}\right|.
\]
In other words, the $[\mathcal{C}_1,(c_1,d_1)]\times [\mathcal{C}_2,(c_2,d_2)]$-entry of $A_{\pi}$ is equal
to the number of pairs of type $(c_2,d_2)$ and of configuration $\mathcal{C}_2$ induced by one pair
$(c_1,d_1)$ from configuration $\mathcal{C}_1$. Now if we have a column-vector
$u=(u_{\mathcal{C}})_{\mathcal{C}\in \Lambda}$, where $u_{\mathcal{C}}(c,d)$ is the number of states of
type $(c,d)$ and of configuration $\mathcal{C}$ that we have at certain level, and we choose $\pi \in \Pi$,
then the number of pairs of each configuration on the next level is given by the vector $A_{\pi}u$. Put
$\mathscr{M}=\{A_{\pi}:\pi \in \Pi \}$.

Now consider all orbits of $a$ on $X^{n}$, take their configurations with respect to $h$ defined up to the
$n$-th level, and define the column-vector $u=(u_{\mathcal{C}})_{\mathcal{C}\in \Lambda}$ as above. To
define the action of $h$ on the $(n+1)$-st level we choose $\pi =(\pi_{\mathcal{C}})_{\mathcal{C}\in
\Lambda}\in \Pi $, and for every orbit with configuration $\mathcal{C}$ we define the action of the states
of $h$ along this orbit using permutation $\pi _{\mathcal{C}}$ as usual (we choose a permutation for every
configuration even if not all configurations appear on the $n$-th level). In this way the conjugator $h$ is
defined up to the $(n+1)$-st level. Then the number of active states of $h$ on the $n$-th level is equal to
$\theta_{\pi}\cdot u$. The vector $v=(v_{\mathcal{C}})_{\mathcal{C}\in \Lambda}$, where
$v_{\mathcal{C}}(c,d)$ is equal to the number of all states of type $(c,d)$ over all orbits of $a$ on
$X^{n+1}$ with configuration $\mathcal{C}$, is equal to $v=A_{\pi}u$.

The process starts at the zero level, where we have the vector $u_{0}=(u_{\mathcal{C}})$ such that
$u_{\mathcal{C}}(e,e)=1$ for the configuration $\mathcal{C}=\{(a,b);DP_{\mathcal{C}}=\{(e,e)\}\}$, which
corresponds to the pair $(a,b)$, and $u_{\mathcal{C}^{\prime}}(c,d)=0$ for all other pairs and
configurations. Then we make choices $\pi_{0},\pi_{1},\ldots ,\pi_{n},\ldots $ from $\Pi$ and construct the
conjugator $h$. It follows from the above discussion that the activity of $h$ can be calculated by the
following rules:
\begin{equation*}
\theta_{n}(h)=\theta_{\pi_{n}}u_{n}\qquad \mbox{ and }\qquad u_{n+1}=A_{\pi_{n}}u_{n}
\end{equation*}
for all $n\geq 0$. If there is a choice such that the sequence $\theta_{n}(h)$ is bounded, then there will
be an eventually periodic choice, and hence the constructed conjugator $h$ will be finite-state and
bounded.

Hence the automorphisms $a$ and $b$ are conjugate in the group $\pol[0]$ if and only if there exists a
sequence $A_{n}\in \mathscr{M}$ such that the corresponding sequence $\theta_{A_{n}}u_{n}$ is bounded. The
last problem is solvable and can be deduced from the result presented in Appendix.

In this method we don't need to solve the auxiliary conjugacy problems in $\mathrm{Pol}({-1})$ as in the
previous method, but the problem reduces to certain matrix problem which should also be solved, while the
previous method was direct. We demonstrate both approaches in Examples~\ref{ex_conj_bounded1} and
\ref{ex_conj_bounded2} of Section~\ref{Section_Examples}.

We note that both approaches also solve the respective simultaneous conjugacy problems. We have proved the following
theorem.

\begin{theorem}
The (simultaneous) conjugacy problem in the group of bounded automata is
decidable.
\end{theorem}

The above methods not only solve the studied conjugacy problem but also provide a construction for a
possible conjugator.

Similarly, one can solve the conjugacy problem for bounded automorphisms in every group $\pol[n]$. However, we have a
stronger statement.

\begin{proposition}
Two bounded automorphisms are conjugate in the group $\pol[\infty] $ if and only if they are conjugate in the group
$\pol[0]$.
\end{proposition}

\begin{proof}
Let $a$ and $b$ be two bounded automorphisms, which are conjugate in $\pol[n]$ for $n\geq 1$. We proceed as
in the first method above. Again the problem reduces to the case when a conjugator $h\in\pol[n]$ lies on a
circuit. Let $u$ be the word which is read along the circuit so that $h|_u=h$. We consider the two cases as
in the proof of correctness of the first approach.

If $u^a=v\neq u$ then the state $h|_v$ should be in $\pol[n-1]$. But then, $h=h|_u=a|_u h|_v
(b|_{u^h})^{-1}\in\pol[n-1]$. Hence, $a$ and $b$ are conjugate in $\pol[n-1]$. The same arguments work if $w^a\neq w$
for some word $w$ of the form $uu\ldots u$.

Suppose $w^a=w$ for every word $w$ of the form $uu\ldots u$. Then $h^{-1}a|_wh=b_{w^h}$. If $a|_w=e$ (and
hence $b|_{w^h}=e$) then define the automorphism $g$ by the rules $g|_w=e$, $g|_v=h|_v$ for all $v\in
X^{|w|}$, $v\neq w$, and the action of $g$ on $X^{|w|}$ is the same as that of $h$. Then $g$ belongs to
$\pol[n-1]$ and it is a conjugator for $(a,b)$.

If $a|_w\neq e$ for every word $w=uu\ldots u$, then some state $a|_w$ is a circuit automorphism and
$a|_w|_v=a|_w$ for some word $v$ of the form $uu\ldots u$. Without loss of generality we can suppose that
$a|_u=a$ and $b|_{u^h}=b$. Then the states $a|_v$ and $b|_{v^h}$ are finitary for all $v\in X^{|u|}$,
$v\neq u$. Consider every orbit $\orb$ of the action of $a$ on $X^{|u|}\setminus u$, let $v\in\orb$ be the
least element in $\orb$ and $m=|\orb|$. Then the finitary automorphisms $a^{m}|_{v}$ and $b^{m}|_{v^h}$ are
conjugate in $\Aut(T)$, and hence they are conjugate in $\pol[-1]$. Define the automorphism $g$ by the
rules: the action of $g$ on $X^{|u|}$ is the same as that of $h$, $g|_u=g$, $g|_v$ is a finitary conjugator
for the pair $(a^{m}|_{v},b^{m}|_{v^h})$, and $g|_{(v)a^i}=\left(a^i|_{v}\right)^{-1} g|_v b^i|_{v^h}$
(also finitary) for every $i=1,\ldots,m-1$ and every orbit $\orb$. Then $g$ is a bounded conjugator for the
pair $(a,b)$.

Inductively we get that $a$ and $b$ are conjugate in $\pol[0]$.
\end{proof}

\section{Examples}\label{Section_Examples}

All examples will be over the binary alphabet $X=\{0,1\}$.

We will frequently use the automorphism $a\in \Aut(T)$ given by the recursion $a=(e,a)\sigma $, where
$\sigma =(0,1)\in \text{Sym}(X)$ is a transposition, which is called the (binary) \textit{adding machine}.
The automorphism $a$ has infinite order, and acts transitively on each level $X^{n}$ of the tree $T$. In
particular, every automorphism which acts transitively on $X^{n}$ for all $n$, is conjugate with $a$ in the
group $\Aut(T)$.

In the next example we investigate the interplay between such properties as being finite-state, contracting, bounded,
polynomial, having or not a finite orbit-signalizer.

\begin{example}
\label{ex_example1} The adding machine $a$ is a bounded automorphism, hence it is contracting and has finite
orbit-signalizer, here $\OS(a)=\{a\}$.

The automorphism $b$ given by the recursion $b=(a,b)$ is finite-state, $\Q(b)=\{e,a,b\}$, $b$ belongs to
$\mathrm{Pol}({1})\setminus \pol[0]$, and $\OS(b)=\{a,b\}$. However $b$ is not contracting, because all elements
$b^{n}=(a^{n},b^{n})$ for $n\geq 1$ are different and would belong to the nucleus.

The automorphisms $b_1=(a,b_2)\sigma $, $b_2=(a,b_1)$ belong to $\mathrm{Pol}({1})\setminus \pol[0]$, but
they have infinite orbit-signalizers. All elements $a^{2n}b_1$ for $n\geq 0$ are different and belong to
the set $\OS(b_1)$. At the same time, $b_1$ and $b_2$ are contracting, for the self-similar group generated
by $a,b_1,b_2$ has nucleus $\mathcal{N}=\{e,a^{\pm 1},b_1^{\pm 1},b_2^{\pm 1},(a^{-1}b_1)^{\pm
1},(a^{-1}b_2)^{\pm 1},(b_1^{-1}b_2)^{\pm 1}\}$.

The automorphism $c=(c,c)\sigma$ is non-polynomial, contracting, and has finite orbit-signalizer, here
$\OS(c)=\{e,c\}$.

The automorphism $d=(d,d^{-2})\sigma $ is contracting, the nucleus of the group $\langle d\rangle $ is
$\mathcal{N}=\{e,d^{\pm 1},d^{\pm 2},d^{\pm 3}\} $. At the same time, the group $\langle a,d\rangle $ is
not contracting; for $da=(da,d^{-2})$ and its powers $(da)^{n}$ are different and would be in the nucleus.

The automorphism $g=(a,g^{2})$ is functionally recursive but not finite-state. Hence the automorphism
$f=(g,g^{-1})\sigma $ is functionally recursive, not finite-state, and has finite orbit-signalizer, here
$\OS(f)=\{e,f\}$.
\end{example}

In the next example we illustrate the solution of the order problem.

\begin{example}\label{ex_order_problem}
Consider the automorphisms $b=(a,b)$ and $a=(1,a)\sigma$. The order graph $\Phi(a)$ is a subgraph of $\Phi(b)$ shown in
Figure~\ref{fig_Order_Graphs}. There is a cycle labeled by $2$, hence $a$ and $b$ have infinite order.

The order graph $\Phi(c)$ for the automorphism $c=(c,\sigma)$ is shown in
Figure~\ref{fig_Order_Graphs}. There are no cycles with labels $>1$, hence $c
$ has finite order, here $|c|=2$.
\end{example}

\begin{figure}
\begin{center}
\psfrag{1}{$1$} \psfrag{2}{$2$} \psfrag{a}{$a$} \psfrag{b}{$b$} \psfrag{c}{$c$} \psfrag{s}{$\sigma$} \psfrag{e}{$e$}
\epsfig{file=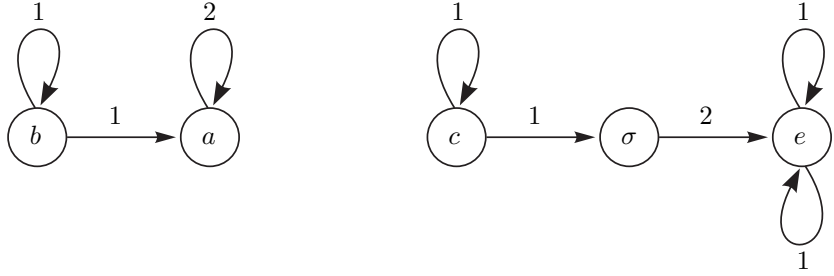,width=110mm} \caption{The order graphs $\Phi(b)$ (on the left) and $\Phi(c)$ (on the
right)}\label{fig_Order_Graphs}
\end{center}
\end{figure}

Let us illustrate the construction of the conjugator graph and basic
conjugators.

\begin{example}\label{ex_conjugacy_problem}
Consider the conjugacy problem for the trivial automorphism $e$ with itself. Here $\OS(e)=\{e\}$ and
$\CPi(e,e)=\Sym(X)=\{\varepsilon,\sigma\}$. The conjugator graph $\Psi(e,e)$ is shown in
Figure~\ref{fig_Conjugacy_Graphs}. There are two defining subgraphs of the graph $\Psi(e,e)$, each consists of the one
vertex $(e,e,\pi)$ with loops in it for $\pi\in\{\varepsilon,\sigma\} $. The corresponding basic conjugators are
$h_1=(h_1,h_1)=e$ and $h_2=(h_2,h_2)\sigma$.

Consider the conjugacy problem for the adding machine $a=(e,a)\sigma$ and its inverse $a^{-1}=(a^{-1},e)\sigma$. Here
$\OS(a)=\{a\}$, $\OS(a^{-1})=\{a^{-1}\}$, and $\CPi(a,a^{-1})=\{\varepsilon,\sigma\}$. There is one orbit of the action
of $a$ on $\{0,1\}$, $a^2|_0=a$ and $a^{-2}|_{0^{\pi}}=a^{-1}$ for every $\pi\in\{\varepsilon,\sigma\}$. The conjugator
graph $\Psi(a,a^{-1})$ is shown in Figure~\ref{fig_Conjugacy_Graphs}. There are two defining subgraphs of the graph
$\Psi(a,a^{-1})$, each consists of the one vertex $(a,a^{-1},\pi)$ with loop in it for $\pi\in\{\varepsilon,\sigma\}$.
The corresponding basic conjugators are $h_1=(h_1,h_1a^{-1})$ and $h_2=(h_2,h_2)\sigma$.

Consider the conjugacy problem for the adding machine $a=(e,a)\sigma$ and the automorphism $b=(e,b^{-1})\sigma$. Here
$\OS(a)=\{a\}$, $\OS(b)=\{b,b^{-1}\}$, and $\CPi(a,b)=\CPi(a,b^{-1})=\{\varepsilon,\sigma\}$. There is one orbit of the
action of $a$ on $\{0,1\}$, $a^2|_0=a$, $b^2|_{0^{\pi}}=b^{-1}$, and $b^{-2}|_{0^{\pi}}=b$ for every
$\pi\in\{\varepsilon,\sigma\}$. The conjugator graph $\Psi(a,b)$ is shown in Figure~\ref{fig_Conjugacy_Graphs}. There
are four defining subgraphs of the graph $\Psi(a,b)$, each consists of the two vertices $(a,b,\pi_1)$ and
$(a,b^{-1},\pi_2)$ with the induced edges for $\pi_1,\pi_2\in\{\varepsilon,\sigma\}$. The corresponding basic
conjugators $h_1,h_2,h_3,h_4$ are defined as follows
\begin{equation*}
\begin{array}{llll}
h_1=(g_1,g_1) \  & h_2=(g_2,g_2) \  & h_3=(g_3,ag_3)\sigma \  &
h_4=(g_4,ag_4)\sigma \\
g_1=(h_1,h_1b) \  & g_2=(h_2,h_2)\sigma \  & g_3=(h_3,h_3b) \  & g_4=(h_4,h_4)\sigma,
\end{array}
\end{equation*}
where $g_1,g_2,g_3,g_4$ are actually the basic conjugators for the pair $(a,b^{-1})$.
\end{example}

\begin{figure}
\begin{center}
\psfrag{eee}{\small $(e,e,\varepsilon)$} \psfrag{ees}{\small $(e,e,\sigma)$} \psfrag{aae}{\small
$(a,a^{-1},\varepsilon)$} \psfrag{aas}{\small $(a,a^{-1},\sigma)$} \psfrag{abe}{\small $(a,b,\varepsilon)$}
\psfrag{ab1s}{\small $(a,b^{-1},\sigma)$} \psfrag{ab1e}{\small $(a,b^{-1},\varepsilon)$} \psfrag{abs}{\small
$(a,b,\sigma)$} \psfrag{0}{\small $0$} \psfrag{0,1}{\small $0,1$} \epsfig{file=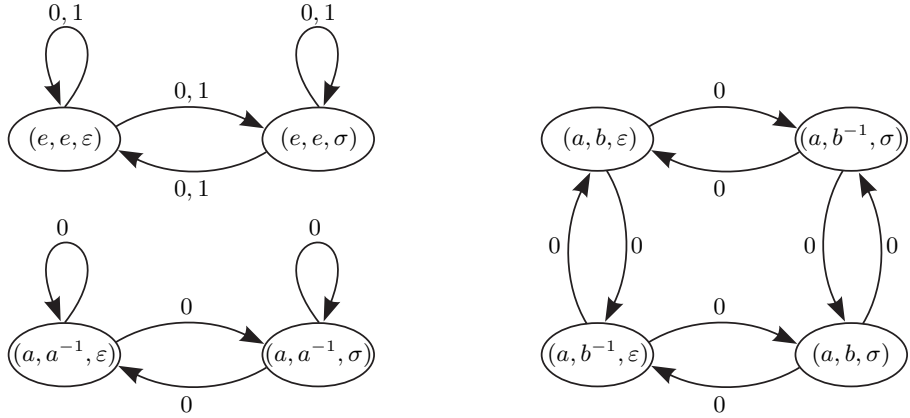,width=120mm} \caption{The
conjugator graph $\Psi(e,e)$ (on the top left), $\Psi(a,a^{-1})$ (on the bottom left), and $\Psi(a,b)$ (on the
right)}\label{fig_Conjugacy_Graphs}
\end{center}
\end{figure}

The next example shows that the condition of having finite orbit-signalizers cannot be dropped in
Theorem~\ref{theor_contr_conjug}, and that Theorem~\ref{thm_bounded_conju_in_Aut_f} does not hold for
polynomial automorphisms.

\begin{example}\label{ex_contr_finiteOS_counter_ex}
Consider the automorphisms $b_1=(a,b_2)\sigma $, $b_2=(a,b_1)$ defined in Example~\ref{ex_example1}.
Inductively one can prove that the state $b_1^{2^{n}}|_{0^{n}}$ is active for every $n$, and hence the
automorphism $b_1$ acts transitively on $X^{n}$ for every $n$. Thus $a$ and $b_1$ have the same orbit types
(see page~\pageref{page_orbit_type}) and therefore they are conjugate in the group $\Aut(T)$. Both $a$ and
$b_1$ are contracting, however, $b_1$ has infinite orbit-signalizer, and hence it is not conjugate with $a$
in the group $\FSG$, by Proposition~\ref{prop_fso_necessary}.
\end{example}

Finally, we illustrate the solution of the conjugacy problem in the group of bounded automata.

\begin{example}\label{ex_conj_bounded1}
Consider the conjugacy problem for the adding machine $a=(e,a)\sigma$ and its inverse $a^{-1}=(a^{-1},e)\sigma$ in the
group of bounded automata.

There are two configurations for the pair $(a,a^{-1})$:
\begin{align*}
\mathcal{C}_1=\{(a,a^{-1}), DP_1=\{(e,e)\}\},\quad \mathcal{C}_2=\{(a,a^{-1}), DP_2=\{(e,e),(e,a^{-1})\}\}.
\end{align*}
Neither of them is satisfied by the trivial automorphism, and hence by a finitary automorphism. In
particular, $a$ and $a^{-1}$ are not conjugate in the group $\mathrm{Pol}({-1})$. The pair $(a,a^{-1})$ is
not detected in Step~1 of the first approach, basically, because $(a,a^{-1})$ is not conjugate in
$\pol[-1]$. There are no $u$ that satisfy Step~2, because $a$ has no fixed vertices. Hence, $a$ and
$a^{-1}$ are not conjugate in the group $\pol[\infty]$.

For the second method we get the choice set $\Pi =\{(\varepsilon ,\varepsilon),(\varepsilon
,\sigma),(\sigma ,\varepsilon),(\sigma ,\sigma)\}$. The configuration $\mathcal{C}_{1}$ induces the
configuration $\mathcal{C}_{2}$ on the next level when we choose the conjugating permutation $\varepsilon
$; here the pair $(e,e)$ induces one pair $(e,e)$ and one pair $(e,a^{-1})$. For the choice $\sigma$, the
configuration $\mathcal{C}_{1}$ induces $\mathcal{C}_{1}$, and the pair $(e,e)$ gives two pairs $(e,e)$.
For the choice $\varepsilon$, the configuration $\mathcal{C}_{2}$ induces $\mathcal{C}_{2}$, here the pair
$(e,e)$ induces one pair $(e,e)$ and one pair $(e,a^{-1})$, and the pair $(e,a^{-1})$ gives two pairs
$(e,a^{-1})$. For the choice $\sigma$, the configuration $\mathcal{C}_{2}$ induces $\mathcal{C}_{2}$, here
the pair $(e,e)$ gives two pairs $(e,e)$, and the pair $(e,a^{-1})$ gives one pair $(e,e)$ and one pair
$(e,a^{-1})$. We get the following set of matrices $A_{\pi}$ and vectors $\theta_{\pi}$:
\begin{align*}
A_{(\varepsilon ,\varepsilon)}&=\left(
\begin{array}{ccc}
0 & 0 & 0 \\
1 & 1 & 0 \\
1 & 1 & 2
\end{array}
\right) ,&A_{(\varepsilon ,\sigma)}=\left(
\begin{array}{ccc}
0 & 0 & 0 \\
1 & 2 & 1 \\
1 & 0 & 1
\end{array}
\right),\\ A_{(\sigma ,\varepsilon)}&=\left(
\begin{array}{ccc}
2 & 0 & 0 \\
0 & 1 & 0 \\
0 & 1 & 2
\end{array}
\right) ,&A_{(\sigma ,\sigma)}=\left(
\begin{array}{ccc}
2 & 0 & 0 \\
0 & 2 & 1 \\
0 & 0 & 1
\end{array}
\right) .
\end{align*}
\begin{equation*}
\theta_{(\varepsilon ,\varepsilon)}=(0,0,1),\quad \theta_{(\varepsilon ,\sigma)}=(0,1,0),\quad \theta_{(\sigma
,\varepsilon)}=(1,0,1),\quad \theta_{(\sigma ,\sigma)}=(1,1,0).
\end{equation*}%
The initial vector is $u_{0}=(1,0,0)^{t}$ and on $n$-th step we get $u_{n+1}=A_{\pi_{n}}u_{n}$ and $\theta_{n}=\theta
_{\pi_{n}}u_{n}$ when we choose $\pi_{n}\in \Pi $. For any choice $\{\pi_{n}\}_{n\geq 0}\subset \Pi $ the sequence
$\theta_{n}$ has exponential growth, and hence $a$ and $a^{-1}$ are not conjugate in the group $\mathrm{Pol}({\infty
})$ of polynomial automata.
\end{example}

\begin{example}\label{ex_conj_bounded2}
Consider the conjugacy problem for the bounded automorphisms $b=(\sigma,b)$ and $c=(c,\sigma)$. Notice that the pairs
$\sigma$, $c$ and $b$, $\sigma$ are not conjugate in $\Aut(T)$. Hence, only $\sigma$ may appear as the action on $X$ of
a possible conjugator, and we take $\CPi(b,c)=\{\sigma\}$. Here $\OS(b)=\{e,\sigma,b\}$ and $\OS(c)=\{e,\sigma,c\}$,
$\CPi(\sigma,\sigma)=\{\varepsilon,\sigma\}$. The configurations for the pair $(b,c)$ are the following:
\begin{align*}
&\mathcal{C}_1=\{(b,c), DP_1=\{(e,e)\}\}, \qquad \mathcal{C}_2=\{(\sigma,\sigma), DP_2=\{(e,e)\}\}, \\
&\mathcal{C}_3=\{(e,e),DP_3=\{(e,e)\}\}.
\end{align*}
Let us check what configurations are satisfied by a finitary automorphism as described after
Corollary~\ref{cor_bound_conj_in_finitary}. The configurations $\mathcal{C}_2$ and $\mathcal{C}_3$ are
satisfied by the trivial automorphism and have depth $0$. For $\pi\in\CPi(b,c)$ we get that the
configuration $\mathcal{C}^{\prime}_{1,\pi}$ induced by $\mathcal{C}_1$ is equal to $\mathcal{C}_1$.
Therefore $\mathcal{C}_1$ is not satisfied by a finitary automorphism, and hence $b$ and $c$ are not
conjugate in $\mathrm{Pol}({-1}) $. In Step~1 of the first approach we detect pairs $(e,e)$ and
$(\sigma,\sigma)$. In Step~2 if we take $u=1$ and $h=(e,h)\sigma$ then $h^{-1}bh=c$. Hence $(b,c)$ is
detected in Step~2 and $b,c$ are conjugate in the group $\pol[0]$.

For the second method, we take for the choice set $\Pi =\{(\sigma ,\varepsilon ,\varepsilon),(\sigma ,\sigma
,\varepsilon),(\sigma ,\varepsilon ,\sigma),(\sigma ,\sigma ,\sigma)\}$. All matrices $A_{\pi}$ are the same for $\pi
\in \Pi $. The vectors $\theta_{\pi}$ are as follows
\begin{equation*}
A_{\pi}=\left(
\begin{array}{ccc}
1 & 0 & 0 \\
1 & 0 & 0 \\
0 & 2 & 2
\end{array}
\right) ,\qquad
\begin{array}{cc}
\theta_{(\sigma ,\varepsilon ,\varepsilon)}=(1,0,0),\  & \theta_{(\sigma
,\sigma ,\varepsilon)}=(1,1,0), \\
\theta_{(\sigma ,\varepsilon ,\sigma)}=(1,0,1),\  & \theta_{(\sigma ,\sigma ,\sigma)}=(1,1,1).
\end{array}
\end{equation*}
The initial vector is $u_{0}=(1,0,0)^{t}$ and $u_{n}=A^{n}u_{0}=(1,1,2^{n}-2) $ independently of our choice. If we
choose $\pi_{n}=(\sigma ,\varepsilon ,\varepsilon)$ for all $n\geq 0$ then the sequence $\theta_{n}=(1,0,0)\cdot
u_{n}=1$ is bounded. Hence $b$ and $c$ are conjugate in the group $\pol[0]$. The conjugator corresponding to our choice
is the adding machine $a$.
\end{example}

\bibliographystyle{plain}

\appendix

\renewcommand{\labelenumi}{\Roman{enumi}.}
\newcommand{\boundedtraj}{{bounded trajectory problem}}

\newcommand{\setvec}{{V}}
\newcommand{\setmat}{{\cal{M}}}
\newcommand{\z}{{\mathbb{Z}}}
\newcommand{\cN}{{\mathbb{N}}}
\newcommand{\vectun}{{e }}
\providecommand{\com[2]}{\begin{tt}[#1: #2]\end{tt}} \providecommand{\comrj[1]}{\com{RJ}{#1}}

\section{On the existence of a bounded trajectory for nonnegative integer systems}%

\begin{center} \large Rapha\"el M. Jungers \end{center}


The purpose of this note is to prove the following theorem.

\begin{theorem}
The following \textit{\boundedtraj} is decidable.\vspace{6pt}

{\sc INSTANCE:} A finite set of nonnegative integer matrices $\setmat=\{A_1,\dots,A_m\}\subset \z^{n\times
n}$ and a finite set of nonnegative integer vectors $\setvec=\{u_1,\dots,u_p\}\subset \z^{n}$.\vspace{3pt}

{\sc PROBLEM:} Determine whether there exists a sequence $(i_t)_{t=1}^\infty,\, i_t \in \{1,\dots,m\}$ and
an \emph{initial vector} $v_0\in \setvec$ such that the sequence of vectors determined by the recurrence
\begin{equation}\label{eq-rec}
v_t=A_{i_t}v_{t-1},\, t=1,2,\dots
\end{equation}
is bounded.

\end{theorem}

\vspace{6pt}

In the following, $ \setmat^*,\setmat^t$ denote respectively the set of all products of matrices in
$\setmat,$ and the set of all products of length $t$ of matrices in $\setmat.$

This problem is closely related to the so called \emph{joint spectral subradius} of a set of matrices,
which is the smallest asymptotic rate of growth of any long product of matrices in the set, when the length
of the product increases. For a survey on the joint spectral subradius and similar quantities, see
\cite{jungers_lncis}. While the joint spectral subradius is notoriously Turing-uncomputable in general, we
will see that in our precise situation, we are able to provide an algorithmic solution to the problem.

The next lemma states that if there is a bounded trajectory, then it can be obtained with an eventually
periodic sequence of matrices.

\begin{lemma}\label{lemma1}
Let $\setmat,\setvec$ be an instance of the \boundedtraj. There exists a sequence $(v_t)$ as given by
Equation (\ref{eq-rec}) which is bounded if and only if there exist matrices $A,B\in \setmat^*$ and a
vector $v_0\in \setvec$ such that the sequence $u_t=A^tBv_0$ is bounded.
\end{lemma}
\begin{proof}
The if-part is obvious. In the other direction, if the set $\{v_t=A_{i_t}\dots A_{i_1}v_0\}$ is bounded it
must be finite. Thus, there actually exist $A,B\in\setmat^*$ such that $v=Bv_0$ and $Av=v$.
\end{proof}

As it turns out it is possible to check in polynomial time, given a nonnegative integer matrix $A$ and a
vector $v,$ whether the sequence $u_t=A^tv$ is bounded.  In fact, as we show below, this does not really
depend on the actual value of the entries of $A$ and $v,$ but only for each entry of $A$ whether it is
equal to zero, one, or larger than one, and for each entry of $v$ whether it is equal to zero or larger
than zero. For this reason we introduce two operators that get rid of the inessential information.

\begin{definition}
Given any nonnegative matrix (or vector) $M\in \mathbb{Z}^{n_1\times n_2}$, we denote by $\sigma(M)$ the
matrix in $\{0,1,2\}^{n_1\times n_2}$ in which all entries larger than two are set to two, while the other
entries are equal to the corresponding ones in $M$.

Similarly, we denote by $\tau(M)$ the matrix in $\{0,1\}^{n_1\times n_2}$ in which all entries larger than
zero are set to one, while the other entries are equal to zero.
\end{definition}

We can now prove the main ingredient of our algorithm.

\begin{theorem}\label{thm-single-matrix}
Given a nonnegative matrix $A\in \mathbb{Z}^{n\times n}$, and two indices $1\leq i,j\leq n,$ the sequence
$(A^t)_{i,j}$  remains bounded when $t$ grows if and only if the sequence $(\sigma(A)^t)_{i,j}$ remains
bounded. Moreover, the boundedness of the sequence $(A^t)_{i,j}$ can be checked in polynomial time.
\end{theorem}

\begin{proof}
We consider the matrix $A$ as the adjacency matrix of a directed graph on $n$ vertices.  The edges of this
graph are given by the nonzero entries of $A$. The graph may have loops, {i.e.}, edges from a node to
itself, which correspond to diagonal entries.  We say that \emph{there is a path (of length $t$) from $i$
to $j$} if there is a power $A^t$ of $A$ such that $(A^t)_{i,j}\geq 1.$ Equivalently, there exist indices
$1\leq i_0,\dots, i_{t}\leq n,$  $i_0=i,\ i_t =j,$ such that for all $0\leq t' \leq t-1,$
$A_{i_{t'},i_{t'+1}}\geq 1$. It is obvious that if there is a path from $i$ to $j,$ then there is such a
path of length less than $n$.

We recall some easy facts from graph theory (see \cite{jungersprotasovblondel06} for proofs and
references). For any directed graph, there is a partition of the set $V$ of its vertices in nonempty
disjoint sets (the \emph{strongly connected components}) $V_1, \ldots , V_I$ such that for all $v,w\in V,$
$v\neq w,$ there is a path from $v$ to $w$ and  a path from $w$ to $v$ if and only if they belong to the
same set in the partition. If there is no path from $v$ to itself, then $\{v\}$ is said to be a
\emph{trivial connected component.} Moreover there exists a (non necessarily unique) ordering of the
subsets in the partition such that for any two vertices $i \in V_k,$ $j \in V_l$, there cannot be a path
from $i$ to $j$ whenever $k
> l$. There is an algorithm to obtain this partition in $O(n)$
operations (with $n$ the number of vertices). In matrix terms, this means that one can find a permutation
matrix $P$ such that the matrix $P^TAP$ is in block upper diagonal form, where each block on the diagonal
corresponds to a strongly connected component.

In the following, we suppose for the sake of clarity that $A$ is already in block triangular shape. It is
clear that entries in the blocks under the diagonal remain equal to zero in any power of $A.$  We need a
different treatment for the entries \emph{within diagonal blocks} and the entries in blocks \emph{above the
diagonal.}
\begin{itemize}
\item {\bf Diagonal blocks.} Let us consider an arbitrary diagonal block $B_l$, which is strongly connected
by definition.  It is easy to see that either all the entries in the block remain bounded or all the
entries are unbounded.  This occurs if and only if the spectral radius of $B_l$ is larger than one. It is
easy to see that given a nonnegative matrix with integer entries whose corresponding graph is strongly
connected, its spectral radius is larger than one if and only if one of these conditions is satisfied:
\begin{itemize}
\item There is an entry in $B_l$ larger than one. \item There is a row in $B_l$ with two entries larger
than zero.
\end{itemize}  Observe that these conditions do only depend on $\sigma(A).$

\item {\bf Non-diagonal blocks.}
Let us consider a particular  $(i,j)$-entry in a non-diagonal block.  We will prove that this entry is
unbounded if and only if one of the following conditions holds (and these conditions can be checked in
polynomial time):
\begin{enumerate}
\item There is a path $(i=i_0,i_1,\dots, i_{t-1},i_t=j)$ from $i$ to $j,$ and one of the entries
$(i_s,i_s)$ is unbounded for $0\leq s \leq t.$

\item \label{cond-ijk} There exists $t$ such that
\begin{equation}
A^t_{i,i}, A^t_{i,j}, A^t_{j,j} \geq 1.
\end{equation}
Moreover, if this condition holds, there is such a $t$ smaller than $n^3$\cite[Proposition
1]{jungersprotasovblondel06}.

\item There exist two indices $i'\neq j'$ such that there is a path from $i$ to $i'$, a path from $j'$ to $j,$
and such that the pair $(i',j')$ satisfies condition II above.
\end{enumerate}
\end{itemize}
It is straightforward to check that any of these three conditions implies that the $(i,j)$-entry is
unbounded.

We {\textbf{claim}} that if the $(i,j)$-entry is unbounded yet I and II fail, then III should hold. We
prove the claim by induction on the number of vertices.  The claim is obvious for $n=1$. Take now an
arbitrary $n,$ and suppose that the claim holds for $n-1.$ We consider an $n$-by-$n$ matrix such that the
$(i,j)$-entry is unbounded, but I and II fail.

First, we must have that either $(A^t)_{i,i}=0$ for all $t$ or $(A^t)_{j,j}=0$ for all $t$. Indeed, it is
not difficult to see that if there exist $ t_1,t_2,t_3$ such that $(A^{t_1})_{i,i}\geq 1,$
$(A^{t_2})_{j,j}\geq 1,$ $(A^{t_3})_{i,j}\geq 1,$ then condition II holds (see \cite[proof of Proposition
1]{jungersprotasovblondel06} for a proof). We thus suppose without loss of generality that $(A^t)_{j,j}=0$
for all~$t$, which means that $\{j\}$ is a trivial connected component. (If it is not the case, then the
proof is symmetrically the same replacing $j$ with $i$).

Now, since $$(A^t)_{i,j}=\sum_{k}A^{t-1}_{i,k}A_{k,j}, $$ it comes that there is an index $k\neq j$ such
that $(A^t)_{i,k}$ is unbounded and $A_{k,j}\geq 1.$ Moreover $k\neq i$ because Condition I does not hold.
Thus, if the pair $(i,k)$ satisfies Condition II the proof is done, because there is a path from $k$ to
$j.$ If not, we now show that one can remove the row and column corresponding to $j$ in the matrix $A$ and
obtain a submatrix $A'$ which fulfills the assumptions of the claim.

Firstly, the entry $(i,k)$ is also unbounded in the powers of $A'$. Indeed, we know that $\{j\}$ is a
trivial component and there is no path from $j$ to $k$. In matrix terms, it means that $A$ can be
block-upper triangularized with the entry corresponding to $k$ before the entry corresponding to $j,$ and
$k,j$ in different blocks. Hence, one can erase all the rows and columns of all blocks after the one
corresponding to $k$ without changing the successive values of the entry $(i,k)$.

Secondly, we just assumed that $(i,k)$ does not satisfy Condition II, and it cannot satisfy Condition I
either, because then Condition I would also hold on $(i,j)$ in the matrix $A$, since there is a path from
$k$ to $j$ in $A$. Thus, one can apply the induction hypothesis and the claim is proved, because, for any
node $j',$ if there is a path in $A'$ from $j'$ to $k,$ there is a path in $A$ from $j'$ to $j$ (obtained
by appending the edge $(k,j)$).

Finally, remark that all the conditions here only depend on which entries are different from zero (since
they amount to check the existence of paths), except for the condition on the boundedness of the
$(i,i)$-entry and the $(j,j)$-entry in Condition I, which is treated in the first part of this proof
(diagonal blocks).
\end{proof}

We are now in position to present our algorithm:

\textbf{Algorithm for solving the \boundedtraj. }\vspace{-0.2cm}

\begin{enumerate}
\item Construct a new instance of the \boundedtraj:
\[
\setmat'=\{\sigma(A):A\in \setmat\} \quad \mbox{ and }\quad \setvec'=\{\tau(v):v \in \setvec\}.
\]
\item REPEAT
\begin{itemize}
\item $\setvec'\leftarrow \setvec'\cup \{\tau(Av):A\in \setmat', v\in \setvec'\}$
\item $\setmat'\leftarrow \setmat'\cup \{\sigma(AB):A\in \setmat', B\in \setmat'\}$
\end{itemize}
UNTIL no new element is added to $\setvec',\setmat'.$

\item For every pair $(A,v)\in \setmat'\times \setvec'$:\\
      IF the sequence $u_t=A^tv$ is bounded, RETURN YES and STOP.

\item RETURN NO.
\end{enumerate}

\begin{theorem} Algorithm is correct and stops in finite time.
\end{theorem}

\begin{proof}
We first show how to implement Line III in the algorithm. For any column corresponding to a nonzero entry
of $v,$ one just has to check whether all the entries of this column remain bounded in the sequence of
matrices $A^t.$ Thanks to Theorem~\ref{thm-single-matrix}, it is possible to fulfill this requirement

By Lemma~\ref{lemma1} we need to check whether there exist $A,B\in \setmat^*$ and $v\in\setvec$ such that
$A^tBv$ is bounded. Note that $A^tBv$ is bounded if and only if $\sigma(A)^t\tau(Bv)$ is bounded. The
finite sets $\{\sigma(A):A\in \setmat^*\}$ and $\{\tau(Bv):\,B\in\setmat^*, v\in\setvec\}$ are precisely
the sets $\setmat'$ and $\setvec'$ obtained after the loop at Line II in the algorithm. Therefore the
algorithm is correct and stops in finite time.
\end{proof}

Let us show that one should not expect a polynomial time algorithm for the problem.
\begin{proposition}\label{prop-np}
Unless $P=NP,$ there is no polynomial time algorithm for solving the \boundedtraj.
\end{proposition}
\begin{proof}
Our proof is by reduction from the {\textit{mortality problem}} which is known to be NP-hard, even for
nonnegative integer matrices \cite[p. 286]{blondel-mortal}.  In this problem, one is given a set of
matrices $\setmat,$ and it is asked whether there exists a product of matrices in $\setmat^*$ which is
equal to the zero matrix.

We now construct an instance $\setmat',\setvec$ of the \boundedtraj{} such that there is a bounded
trajectory for this instance if and only if the set $\setmat$ is mortal: take $\setmat'=\{A'=2A:A\in
\setmat\}$ and $v_0=\vectun$ (the "all ones vector") as the unique vector in $\setvec.$

Now, it is straightforward that there exists a sequence $(i_t)_{t=1}^\infty$, $i_t \in \{1,\dots,m\}$, such
that the sequence of vectors $$A'_{i_t}\dots A'_{i_1}\vectun=2^tA_{i_t}\dots A_{i_1}\vectun$$ is bounded if
and only if the set $\setmat$ is mortal. Indeed, the matrices in $\setmat$ have nonnegative integer
entries, and if the vector $A_{i_t}\dots A_{i_1}\vectun$ is different from zero, then its (say, Euclidean)
norm is greater or equal to one.
\end{proof}

Also, if one relaxes the requirement that the matrices and the vectors are nonnegative, then the problem
becomes undecidable, as shown in the next proposition.
\begin{proposition}
The \boundedtraj{} is undecidable if the matrices and vectors in the instance can have negative entries.
\end{proposition}
\begin{proof} (sketch)
It is known that the mortality problem with entries in $\z$ is undecidable \cite[Corollary
2.1]{jungers_lncis}.  We reduce this problem to the \boundedtraj{} in a way similar as in Proposition
\ref{prop-np}, except that we build much larger matrices: we make $2^n$ copies of each matrix in $\setmat$
and place them in a large block-diagonal matrix. That is, our matrices in $\setmat'$ are of the shape
\[
\{\mbox{diag}(2A,2A,\dots, 2A):A\in \setmat\}.
\]
Now we take $V=\{v_0\},$ where $v_0\in\{-1,1\}^{2^nn}$ is the concatenation of all the different
$n$-dimensional $\{-1,1\}$-vectors. This vector has a bounded trajectory if and only if there exists a zero
product in $\setmat^*.$
\end{proof}

\newcommand{\noopsort}[1]{} \newcommand{\singleletter}[1]{#1}

\end{document}